\documentclass[a4paper, reqno, 14pt]{amsart}
\usepackage[usenames,dvipsnames]{color}

\usepackage{amsthm,amsfonts,amssymb,amsmath}
\usepackage[all]{xy}
\usepackage{xr-hyper}
\usepackage[colorlinks=
   citecolor=Black,
   linkcolor=Red,
   urlcolor=Blue]{hyperref}
\usepackage{verbatim}
\usepackage{pdfsync}

\usepackage[margin=1.5in]{geometry}

\newcommand{\BA}{{\mathbb {A}}}

\newcommand{\BC}{{\mathbb {C}}}

\newcommand{\BF}{{\mathbb {F}}}
\newcommand{\BG}{{\mathbb {G}}}

\newcommand{\BM}{{\mathbb {M}}}

\newcommand{\BP}{{\mathbb {P}}}
\newcommand{\BQ}{{\mathbb {Q}}}
\newcommand{\BR}{{\mathbb {R}}}
\newcommand{\BS}{{\mathbb {S}}}

\newcommand{\BX}{{\mathbb {X}}}

\newcommand{\BZ}{{\mathbb {Z}}}

\newcommand{\CL}{{\mathcal {L}}}
\newcommand{\CM}{{\mathcal {M}}}
\newcommand{\CN}{{\mathcal {N}}}
\newcommand{\CO}{{\mathcal {O}}}

\newcommand{\CS}{{\mathcal {S}}}

\newcommand{\disc}{{\mathrm{disc}}}

\newcommand{\End}{{\mathrm{End}}}

\newcommand{\Gal}{{\mathrm{Gal}}}
\newcommand{\GL}{{\mathrm{GL}}}

\newcommand{\Hom}{{\mathrm{Hom}}}

\renewcommand{\Im}{{\mathrm{Im}}}

\newcommand{\inv}{{\mathrm{inv}}}

\newcommand{\Ker}{{\mathrm{Ker}}\,}

\newcommand{\Lie}{{\mathrm{Lie}}\,}

\newcommand{\PGL}{{\mathrm{PGL}}}

\newcommand{\Spec}{{\mathrm{Spec}}\,}
\newcommand{\Spf}{{\mathrm{Spf}}\,}

\newcommand{\Sp}{{\mathrm{Sp}}}

\newcommand{\tr}{{\mathrm{tr}}\,}

\newcommand{\CCO}{O}

\newcommand{\wh}{\widehat}

\newcommand{\incl}{\hookrightarrow}
\newcommand{\lra}{\longrightarrow}

\newcommand{\lto}{\longmapsto}
\newcommand{\bs}{\backslash}

\newcommand{\uF}{\underline{F}}  
\newcommand{\ep}{\varepsilon}

\newcommand{\nass}{\noalign{\smallskip}}
\newcommand{\htt}{h}

\newtheorem{theorem}{Theorem}[section]
\newtheorem{proposition}[theorem]{Proposition}
\newtheorem{lemma}[theorem]{Lemma}

\newtheorem{corollary}[theorem]{Corollary}

\theoremstyle{definition}
\newtheorem{definition}[theorem]{Definition}

\newtheorem{remark}[theorem]{Remark}

\newtheorem {question}[theorem]{Question}

\numberwithin{equation}{section}

\begin{document}
%------------------------------------------------------

%\thanks{Research of Rapoport and Terstiege partially supported by SFB/TR 45 ``Periods, Moduli Spaces and 
%Arithmetic of Algebraic Varieties" of the DFG. Research of Zhang partially supported by NSF grant DMS 1204365.}

\title{New cases of $p$-adic uniformization}
\author{Stephen Kudla and Michael Rapoport}

\date{\today}
\maketitle

\centerline{\it To G. Laumon on his 60th birthday}

\tableofcontents

\section{Introduction} 
The subject matter of $p$-adic uniformization of Shimura varieties starts with Cherednik's 
paper \cite{Ch} in 1976, although a more thorough historical account would certainly involve at least  the 
names of Mumford and Tate. Cherednik's theorem states that the Shimura curve associated to a 
quaternion algebra $B$ over a totally real field $F$ which is split at precisely one archimedean
place $v$ of $F$ (and ramified at all other archimedean places), and is ramified at a non-archimedean place $w$ of residue characteristic $p$
admits $p$-adic uniformization by the Drinfeld halfplane associated to $F_w$, provided that the
level structure is prime to $p$. In adelic terms, this theorem may be formulated more precisely as follows.

Let $C$ be an open compact subgroup of $(B \otimes_F \BA_F^\infty)^\times$ of the form
\begin{equation*}
C = C^w \cdot C_w,
\end{equation*}
where $C_w \subset (B \otimes_F F_w)^\times$ is maximal compact and $C^w \subset (B \otimes_F \BA_F^{\infty,w})^\times$.
Let $\CS_C$ be the associated Shimura curve.  It has a canonical model over $F$ and its set of complex points, for the $F$-algebra 
structure on $\mathbb C$ given by $v$, has a complex uniformization 
\begin{equation*}
\CS_C(\mathbb C) = B^\times \bs \big[ {\bold X} \times (B \otimes_F \BA_F^\infty)^\times / C\big ],
\end{equation*}
where ${\bold X} = \BC \smallsetminus \BR$, which is acted on by 
$(B \otimes_F F_\BR)^\times$ via a fixed isomorphism $B^\times_v \simeq \rm GL_2(\BR)$.  

Cherednik's theorem states that, 
after extending
scalars from $F$ to ${\bar F}_w$,  there is an isomorphism of algebraic curves over ${\bar F}_w$,
\begin{equation}\label{Chered}
\CS_C \otimes_F {\bar F}_w \simeq (\bar B^\times \bs \big [ \Omega^2_{F_w} \times 
(B \otimes_F \BA_F^{\infty})^\times / C)\big ]) \otimes _{{F_w}} {\bar F}_w , 
\end{equation}
where $\bar B$ is the quaternion algebra over $F$, with the same invariants as $B$, except at $v$ and
$w$, where they are interchanged. Here $\Omega^2_{F_w}$ is the rigid-analytic space $\BP^1_{F_w} \smallsetminus \BP^1(F_w)$ over ${F_w}$ ({\it Drinfeld's halfspace}). This isomorphism is to be interpreted as follows. 

The rigid-analytic space $\bar B^\times \bs \big [ \Omega^2_{F_w} \times 
(B \otimes_F \BA_F^{\infty})^\times / C)\big ]$ corresponds to a unique   projective algebraic curve over $\Spec\,{F_w}$ under the GAGA functor.   In the right hand side of \eqref{Chered}, we implicitly replace the rigid-analytic space  by this projective scheme; extending scalars, we obtain a projective algebraic curve over  ${\bar F}_w$. The statement of Cherednik's theorem is that there exists an isomorphism between these two algebraic curves over ${\bar F}_w$.

Drinfeld \cite{Dr} gave a moduli-theoretic proof of Cherednik's theorem in the special case $F=\BQ$. 
Note that it is only in this case that the Shimura curve considered by Cherednik represents a moduli problem 
of abelian varieties. Furthermore, Drinfeld proved an `integral version' of this theorem which has the original version as a corollary. In his formulation appears  the formal scheme $\wh \Omega^2_{F_w}$ over $\Spec\,\CCO_{F_w}$, with
''generic fiber`` equal to $ \Omega^2_{F_w}$, defined by Mumford, Deligne and Drinfeld. In particular,
he interpreted the formal scheme $\wh \Omega^2_{F_w}$, and its higher-dimensional versions $\wh \Omega^n_{F_w}$
as  formal moduli spaces of {\it special formal $\CCO_{B_w}$-modules}, where $B_w$ is the central division algebra 
over $F_w$ with invariant $1/n$.

This integral uniformization  theorem was generalized to higher-dimensional cases in \cite{RZ}. In these cases, one uniformizes Shimura 
varieties associated to certain unitary groups over a totally real field $F$ which at the archimedean 
places have signature $(1,n-1)$ at one place $v$, and signature $(0,n)$ at all others, and such that the
associated CM-field $K$ has two distinct places over the $p$-adic place $w$ of $F$. (One has to be much, much more
specific to force $p$-adic uniformization, cf. loc.~cit. pp. 298--315). Using these methods, Boutot and Zink \cite{BZ} have given a 
conceptual proof of Cherednik's theorem for general totally real fields, and constructed at the same time  integral models for the corresponding Shimura varieties. Such integral models were also constructed for general Shimura curves by Carayol \cite{C}. In this context also falls the work of Varshavsky 
\cite{V, V2}, which concerns the $p$-adic uniformization of Shimura varieties associated to similar unitary groups, again where the $p$-adic
place $w$ splits in $K$ (but not the construction of integral models).

In this paper, we give a new (very restricted) class of Shimura varieties which admit $p$-adic uniformization. For  this class we prove  
$p$-adic uniformization for their generic fibers and, in certain cases, %for some members in this class, 
also $p$-adic uniformization for their integral models. 
The relevant reductive groups are defined in terms of two dimensional hermitian spaces for CM fields and 
the corresponding Shimura varieties represent
a moduli problem of abelian varieties with additional structure.   By extending  the moduli problem integrally, we obtain integral models of these 
Shimura varieties which allow us to formulate and prove an `integral' version of our uniformization result.
Behind this integral version of Theorem~\ref{ThmA} is our interpretation of the Drinfeld formal halfplane $\wh{\Omega}^2_F$, for a $p$-adic 
local field $F$ and a quadratic extension $K$ of $F$, as the formal moduli space of polarized two-dimensional $\CCO_K$-modules of Picard type, 
established in a previous paper\footnote{In \cite{KR}, we excluded ramification in the case of even residue characteristic, and this 
restriction will thus be in force in the present paper. }.

The simplest
example is the following. 
Let $K$ be an imaginary quadratic field, and let $V$ be a hermitian vector space of dimension $2$ over $K$
of signature $(1,1)$. Let $G = {\rm GU}(V)$ be the group of unitary 
similitudes of $V$.
For $C \subset G(\BA^\infty)$ an open compact subgroup, there is a Shimura variety ${\rm Sh}_C$ with canonical model over $\mathbb Q$
whose complex points are given by 
$${\rm Sh}_C(\mathbb C) \simeq G(\BQ) \bs [{\bold X} \times G(\BA^\infty)/C],$$
where again ${\bold X} = \BC \smallsetminus \BR$, which is acted on by $G(\BR)$ via a fixed isomorphism 
$G_{\rm ad} (\BR) \simeq \PGL_2(\BR)$.

Suppose that $p$ is a prime that does \emph{not} split in $K$ and that $p\neq 2$ if $p$ is ramified in $K$. 
Suppose that the local
hermitian space $V \otimes_\BQ \BQ_p$ is \emph{anisotropic} and that $C$ has  
%Let $G = GU(V)$ be the group of unitary 
%similitudes of $V$.
%Let $C \subset G(\BA^\infty)$ be an open compact subgroup 
 the form $C = C^p \cdot C_p$, where $C_p$ is the \emph{unique}
maximal compact subgroup of $G(\BQ_p)$. 
Let $\bar V$ be the hermitian space over $K$ which is positive definite, split at $p$, and locally coincides 
with $V$ at all places $\neq \infty,p$, and let $I = {\rm GU}(\bar V)$ be the corresponding group of unitary 
similitudes. Then there is an identification of the adjoint group $I_{\rm ad}(\BQ_p)$ with $\PGL_2(\BQ_p)$ and an action of 
$I(\BQ)$ on $G(\BA^{\infty})/C$. 
\begin{theorem}\label{ThmA}
 There is an
isomorphism of algebraic curves over the completion of the maximal unramified extension $\breve{{\BQ}}_p$ of
$\BQ_p$,
\begin{equation*}
{\rm Sh}_C \otimes_\BQ \breve{{\BQ}}_p
\simeq \big(I(\BQ)\bs [{\Omega}^2_{\BQ_p} \times  G(\BA^{\infty})/ C]\big)\otimes_{\BQ_p}\breve{{\BQ}}_p
\end{equation*}

\end{theorem}

%We also prove a generalization to  higher dimensions, cf. Corollary \ref{rigunif}.  The statement in Theorem \ref{ThmA} looks formally very similar to Cherednik's. However, the Shimura varieties here are not the same (although the corresponding algebraic groups have isomorphic adjoint groups); in particular, our Shimura varieties represent
%a moduli problem of abelian varieties with additional structure, whereas this is not true of Cherednik's Shimura varieties if $F\neq \BQ$. 
%
%By extending  the moduli problem integrally, we obtain integral models of these Shimura varieties. This gives us the possibility of formulating and proving an `integral' version of this theorem, which is also valid  in higher dimensions. Behind this integral version of Theorem~\ref{ThmA} is our interpretation of the Drinfeld formal halfplane $\wh{\Omega}^2_F$, for a $p$-adic local field $F$ and a quadratic extension $K$ of $F$, as the formal moduli space of polarized two-dimensional $\CCO_K$-modules of Picard type, established in a previous paper \cite{KR}. 

Next we give a simplified version of our main theorem about integral uniformization\footnote{Unexplained terms in the statement are defined in the main body of the text. }. 
%is as follows\footnote{Unexplained terms in the statement are defined in the main body of the text. }.
Let $K$ be a CM quadratic extension of a totally real field $F$ of degree $d$ over $\BQ$, and let $V$ be a hermitian vector space of dimension $2$ over $K$
with signature (1, 1) at every archimedean place of $F$. Let $G$ be the group of unitary similitudes of $V$ with multiplier in $\BQ^\times$. 
For an open compact subgroup $C$ of $G(\BA^\infty)$, let ${\rm Sh}_C$ be the canonical model of the corresponding Shimura variety, which, in this case,  
is a projective variety of dimension $d$ defined over $\BQ$. 
 
Suppose that $p$ is a prime
that decomposes completely in $F$ and such that each prime divisor ${\bf p}$ of $p$ in $F$ is inert or ramified in $K$. 
If $p=2$, then suppose that no ${\bf p}$ is ramified in $K$.
We assume that $\inv_{{\bf p}}(V)=-1$ for all 
${\bf p} \vert p$. 
Let $C^p$ be an open compact subgroup of $G(\BA^{\infty, p})$, and let $C=C^p\cdot C_p$, where $C_p$ is the unique maximal compact subgroup of $G(\BQ_p)$. 
%Let ${\rm Sh}_C$ be the canonical model of the corresponding Shimura variety. It is a projective variety of dimension $d$ defined over $\BQ$. 

In section 4, we define a model $\CM_{r, h, V}(C^p)$ of ${\rm Sh}_C$ over $\BZ_{(p)}$ as a moduli space of  \emph{almost principal CM-triples
$(A, \iota, \lambda)$ of generalized CM-type $(r, h)$} 
 with level-$C^p$-structure. Here $r_{\varphi} = 1$, for all complex embeddings $\varphi$ of $K$, and the function $h$ which 
 describes the kernel of the polarization $\lambda$
also has to satisfy the compatibility condition of Proposition \ref{prop.nonempty}. In particular, we demand that 
the localization of the kernel of the polarization $\lambda$ at any ${\bf p}\vert p$  satisfies 
\begin{equation}
\begin{aligned}
 p\cdot\big(\Ker\,\lambda\big)_{\bf p}&=(0),\\ 
\vert (\Ker\,\lambda)_{\bf p}\vert&= \begin{cases} p^2 &\text{when ${\bf p}\vert p$ is unramified,}\\
\noalign{\smallskip}
1&\text{when ${\bf p}\vert p$ is ramified.}
\end{cases}
 \end{aligned}
 \end{equation}
In addition, for each place $v$ of $F$, the invariant $\inv_v(A,\iota,\lambda)$,  
defined in section 3, is required to coincide with the invariant $\inv_v(V)$ of the hermitian space $V$-- see section \ref{section3} for the precise definitions.  
We denote by $\CM_{r,h,V}(C^p)^{\wedge}$
the completion of this model along its special fiber. 

\begin{theorem}\label{thm1.2} There is a 
$G(\BA^{\infty, p})$-equivariant isomorphism of $p$-adic formal schemes
\begin{equation*}
 \CM_{r,h,V}(C^p)^{\wedge} \times _{{\rm Spf}\,\BZ_p}{\rm Spf}\,\breve{\BZ}_p  \simeq   I(\BQ)\bs\big[\big((\widehat{\Omega}^2_{\BQ_p})^d \times_{{\rm Spf}\,\BZ_p}{\rm Spf}\,\breve{\BZ}_p\big) \times
 G(\BA^{\infty})/C\big]  \, .
\end{equation*}
Here $I(\BQ)$ is the group of $\BQ$-rational points of the inner form $I$ of $G$ such that $I_{\rm ad} (\BR)$ is compact, 
$I_{\rm ad}(\BQ_p) \simeq {\rm PGL}_2(\BQ_p)^d, \text{and } I(\BA^{\infty, p}) \simeq G(\BA^{\infty, p})$.

The natural descent datum on the LHS induces on the RHS the natural descent datum on the first factor multiplied
with the translation action of
$ (1, t)$ on $ G(\BA^{\infty})/C=G(\BA^{\infty, p})/C^p \times G(\BQ_p)/C_p$, where $t\in G(\BQ_p)$ is any element with ${\rm ord}_p\,c(t) = 1$ 
for $c: G(\BQ_p)\rightarrow \BQ_p^\times$ the scale homomorphism.
\end{theorem}

In our general integral uniformization result, Theorem~\ref{intunif}, several other types of local conditions are also allowed at the primes 
dividing $p$, conditions which we refer to as uniformizing data of the second and third kind. Both of these are analogous to the 
conditions already considered in Chapter 6 of \cite{RZ}. On the other hand, the uniformizing data of the first kind, as illustrated in the 
two examples just given, is new and arises from the identification of \cite{KR}. 
This identification is valid for any $p$-adic field  (with, as usual, a caveat when $p=2$).  It is remarkable that, for the
$p$-adic uniformization of our class of Shimura varieties, only the case where all localizations $F_{\bf p}$ are equal to $\BQ_p$ is relevant. This can be traced to the fact that, on the one hand the formal 
moduli problem for  $\wh{\Omega}^2_{F_{\bf p}}$ imposes that $\CCO_{F_{\bf p}}$ acts through the structure morphism on the Lie algebra of the formal groups appearing, 
but that,  on the other hand,  the Lie algebras of the relevant abelian varieties are very often free $\CCO_F\otimes\CO_S$-modules, locally on $S$. In fact, the latter condition holds if $F$ is unramified over $p$ and $S$ is a flat $\CO_{E_p}$-scheme. 
 In contrast, in Cherednik's theorem there are no hypotheses on the local extension $F_w/\BQ_p$.  It seems very likely that our uniformization theorem is valid also for non-trivial local extensions. This would require a generalization of our interpretation of the Drinfeld formal halfplane 
in \cite{KR}. Such a generalization is the subject of ongoing work with Th. Zink, comp. Remark \ref{travdeZink}.

In \cite{RZ} a general uniformization theorem valid for arbitrary Shimura varieties of PEL-type is proved. However, in this generality, one only obtains uniformization along the \emph{basic locus} in the special fiber. As soon as this basic locus has dimension
strictly smaller than that of the whole special fiber, the uniformizing formal scheme is no longer $p$-adic; only when all points of the special fiber are basic can there be $p$-adic uniformization. It is then a matter of experience that, in these very rare cases, the uniformizing formal scheme is always a product of Drinfeld halfspaces. This is predicted in \cite{R}, and is also supported by the classification of Kottwitz of \emph{uniform pairs} $(G, \mu)$,
cf. \cite{K2}, \S 6.

In this paper, we are dealing with $p$-divisible groups, say over an algebraically closed field $k$ of characteristic $p$,  equipped with some complex multiplication and with a  compatible polarization, and their associated Dieudonn\'e modules. The (rational)  Dieudonn\'e modules of these $p$-divisible groups inherit these additional structures. 
Here we treat the theory of these Dieudonn\'e modules with additional structure on the most elementary level, not on the group-theoretical level. This is in contrast to Kottwitz's approach where one first fixes a suitable algebraic group $G$ over $\BQ_p$ and then describes these (rational) Dieudonn\'e modules as elements of $G(W_{\BQ}(k))$. We refer to \cite{DOR}, ch. XI, \S 1 for the theory of {\it augmented group schemes} with values in the tannakian category of isocrystals over $k$ which links these two approaches. We feel that Kottwitz's approach  (or that of augmented group schemes) would have been unnatural in our context, although it could most probably be used as an alternative method to obtain our results. As an example, one can compare the proof of Lemma \ref{uniformity} using our method, and the proof of the same statement given in Remark \ref{compkottw} using Kottwitz's method.   In particular, as G.~Laumon pointed out to us, it is quite likely that there is a close connection between the local invariants defined in section \ref{section3} and the Kottwitz invariant of \cite{K}.

In section 3 of \cite{BZ}, Boutot and Zink give a new proof of Cherednik's theorem for an arbitrary totally real field $F$ by embedding the Shimura curve attached to the quaternion algebra  $B$ into a Shimura variety for a twisted unitary similitude group $G^{\bullet}$  associated to $B$ and a CM quadratic extension
$K$ of $F$ (see   section~\ref{append} for the relation between quaternion algebras and twisted unitary similitude groups).  Their proof is analogous to Drinfeld's proof in the case $F=\BQ$,  in that the Shimura variety for $G^{\bullet}$ represents a moduli problem 
of abelian varieties. In fact, this moduli problem  has an integral extension, which provides a  natural integral model; by (a slight generalization of) the results of \cite{RZ}, one obtains 
an integral $p$-adic uniformization theorem. It is essential in this construction that the prime ${\bf p}$ of $F$ at which 
uniformization occurs is {\it split} in the extension $K/F$ and ramified in $B$. In particular, the quaternion algebra $S = B\otimes_FK$ remains a division algebra
so that the group $G^{\bullet}$ is a {\it twisted} unitary similitude group.  
The groups $G$ that we consider in the present paper are analogously determined by an indefinite division quaternion algebra $B$ and a CM extension $K$ of $F$. 
However, the essential distinction is that, in the case of a uniformizing prime ${\bf p}$ of the first kind, we assume that ${\bf p}$ is ramified in $B$ but
does {\it not split} in $K$.  
In fact, for simplicity, we have restricted to the case in which the extension $K$ splits $B$, so that our group $G$ is the (untwisted) group of unitary 
similitudes %\footnote{We have an added condition on the scale factor, which does not play a very essential role.} 
of a $2$-dimensional 
hermitian space $V$ over $K$. By embedding the Shimura curve into the Shimura 
variety attached to $G$,  it should be possible to carry over the Boutot-Zink proof of Cherednik's  theorem using our type of uniformization. 

In general, suppose that $G$ is the twisted unitary similitude group attached to  
an indefinite
division quaternion algebra $B$ and a CM field $K$ over $F$ with associated Shimura data determined by $\bold h$ as in (\ref{defofh}). 
For a suitable open compact subgroup $C$, the corresponding Shimura variety then represents a moduli problem for abelian varieties 
and one expects that it will admit a good integral model and $p$-adic uniformization, both integral and rigid analytic,  
under the conditions described in section~6 of this paper (with the generalization of the notion of a uniformizing prime of the first kind explained in Remark \ref{travdeZink}).  
Of course, there are many other Shimura data, for example those of Cherednik, that induce identical  Shimura data on the associated adjoint groups. 
The resulting Shimura varieties are by no means identical (recall the difference between Shimura varieties of Hodge type and Shimura varieties of abelian type); however,  
it should be possible to obtain integral models and $p$-adic uniformization for them by pullback from those for $G$. 
For the methods employed here and in \cite{BZ}, the (twisted) unitary similitude groups are the fundamental objects. 

%Since the Shimura data of the Cherednik examples and of  our class of examples induce identical  Shimura data on the associated adjoint groups, these two classes are very similar. However, the two classes  are by no means identical (this is similar to the difference between Shimura varieties of Hodge type and Shimura varieties of abelian type). 
%
%

%In fact, there are some basic differences between the statements of Cherednik's theorem (\ref{Chered}) and our Theorem \ref{ThmA}
%and its generalizations. 
%First of all, as noted before, our Shimura varieties always represent a moduli problem for abelian varieties, whereas this is only true for Cherednik's 
%Shimura curves when $F=\BQ$. In fact, for our class of examples, there is a natural way to extend this moduli problem of abelian varieties over the integers and our integral uniformization theorems compare this integral structure with the integral structure defined by the formal scheme version of Drinfeld's halfspace. 
%Another interesting point is that in  Cherednik's theorem there are no hypotheses on the local extension $F_w/\BQ_p$, whereas we have to assume that the local extension is trivial at $p$. However, it seems very likely that our uniformization theorem is valid also for non-trivial local extensions. This would require a generalization of our interpretation of the Drinfeld formal halfplane 
%in \cite{KR}. Such a generalization is the subject of ongoing work with Th. Zink, comp. Remark \ref{travdeZink}. 

The Cherednik-Drinfeld uniformization theorem was used in arithmetic applications, like level-raising, resp. level-lowering of modular forms and also in bounding the size of Selmer groups, cf. \cite{Ri}, esp. \S 4, and \cite{Ra}, \cite{N}. It is to be hoped that similar applications can be found for our uniformization theorem. 

We finally summarize the contents of the various sections. In section 2 we introduce the stack of CM-triples of a fixed {\it generalized CM-type} of arbitrary rank. In section 3 we introduce the local invariants, with values  in $\{\pm 1\}$,  of a CM-triple, one for each place, when the rank is  even. An interesting question that arises in this context is when these local invariants satisfy the product formula, cf. Question \ref{conj.prodform}. We then specialize to rank $2$, and show in section  4 that fixing the local invariants gives a decomposition of the stack of {\it almost principal CM-triples} into stacks with good finiteness properties, and with generic fiber equal to a Shimura variety. In section 5 we consider the local situation, i.e.,  consider $p$-divisible groups instead of abelian varieties, and exhibit conditions on {\it local CM-triples} that guarantee that they are all isogenous to each other and supersingular. In section 6, the results of section 5 are  then used in order to prove an integral $p$-adic uniformization theorem. In section 7 we give a rigid-analytic   uniformization theorem,  which allows us to also treat level structures
that are no longer maximal at $p$. The appendix, section 8, explains the relation between quaternion algebras and twisted unitary similitude groups.

We thank J.~Tilouine for raising the question of the global consequences of our local theorem in \cite{KR}. This paper arose in trying to answer his query. We also thank the referee for his remarks.

\medskip

\noindent{\bf Notation.}  

For a number field $F$, $\Sigma(F)$ (resp. $\Sigma_f(F)$, resp. $\Sigma_\infty(F)$) denotes the set of places (resp. finite places, resp. infinite places). By $\bar\BQ$ we denote the field of algebraic numbers in $\BC$. For any $p$-adic field $F$, we denote by $\breve F$ the completion of the maximal unramified extension of $F$, and by $\CCO_{\breve F}$ its ring of integers. For a perfect field $k$, we write $W(k)$ for its ring of Witt vectors, and $W(k)_\BQ$ for its field of fractions.

\section{Generalized CM-types}

Let $K$ be a CM-field, with totally real subfield $F$. 
\begin{definition}
 Let $n \geqslant 1$. $A$ \emph{generalized CM-type of rank} $n$ for $K$ is a function
\begin{equation*}
r : \Hom_\BQ(K, \bar \BQ) \longrightarrow \BZ_{\geqslant 0},\,\,\quad \varphi \mapsto r_\varphi,
\end{equation*}
such that
\begin{equation*}
r_\varphi + r_{\bar \varphi} = n, \quad\forall \varphi \in \Hom_\BQ(K, \bar \BQ).
\end{equation*}
\end{definition}
We note that the values of $r$ are integers in the interval $[0,n]$. For $n=1$,  this notion
reduces to the usual notion of a CM-type for $K$, i.e., a half-system of complex embeddings. %: a subset $\Phi \subset \Hom_\BQ(K, \bar\BQ)$.
%such that, for every $\varphi \in \Hom_\BQ(K,\bar \BQ)$, precisely one of $\varphi, \bar \varphi$ belongs to
%$\Phi$).
The notion is equivalent to that of a  {\it effective $n$-orientation} of $K$ discussed in \cite{GGK}, V. A, p. 190. 

A generalized CM-type $r$ determines its \emph{reflex field}, the subfield 
$E=E(r)$ of $ \bar \BQ$ characterized by 
\begin{equation*}
 \Gal(\bar \BQ/E) = \{ \sigma \in \Gal (\bar \BQ/\BQ) \mid
r _{\sigma \circ \varphi} = r_\varphi,\, \forall \varphi \}.
\end{equation*}
We will be interested in abelian varieties with action by $\CCO_K$ such that the $\CCO_K$-action on the 
Lie algebra is given by a generalized CM-type.
\begin{definition}
 Let $r$ be a generalized CM-type of rank $n$ for $K$ with associated reflex field $E$. An
abelian scheme $A$ over an $\CCO_E$-scheme $S$ is  \emph{of CM-type} $r$ if $A$ is equipped 
with an $\CCO_K$-action $\iota$ such that,  for all $a\in\CCO_K$,
\begin{equation}\label{signature.condition}
 {\rm char} (T, \iota (a) \vert \Lie  A) = i\big(\underset{\varphi \in \Hom(K, \bar \BQ)} \prod (T-\varphi(a))^{r_\varphi}\big),
\end{equation}
where $i:\CCO_E\rightarrow \mathcal O_S$ is the structure homomorphism ({\it Kottwitz condition}). 
\end{definition}

It will sometimes be convenient to fix a  CM-type  $\Phi\subset \Hom_\BQ(K,\bar\BQ)$ 
and to express  
the function $r$ as a signature
\begin{equation}\label{signature}
((r_v,s_v))_{v\in \Sigma_\infty(F)}, \qquad r_v = r_{\varphi}, \ s_v = r_{\bar\varphi},  \ \varphi=\varphi_v\in \Phi.
\end{equation} 
Here $\varphi=\varphi_v$ induces the place $v\in\Sigma_\infty(F)$. 
In particular, we will sometime refer to (\ref{signature.condition}) as the {\it signature condition}.

Let $(A,\iota)$ be an abelian scheme of CM-type $r$ over an $\CCO_E$-scheme $S$. We will consider polarizations 
$\lambda : A \rightarrow A^\vee$ such that the corresponding Rosati-involution induces the complex
conjugation  on $K$.  Such triples $(A, \iota, \lambda)$ will be called \emph{triples of CM-type $r$}, and $K$ will be clear from the context.

Let $\CM_r$ be the stack of triples of CM-type $r$. 
We note that $\CM_r$  is a Deligne-Mumford stack locally of finite type over $\Spec \, \CCO_E$, where $E=E(r)$ is the 
reflex field of the generalized CM-type $r$. However, it is highly reducible. We will try to separate connected components by introducing additional invariants.  In fact, we will be mostly interested in triples of the following kind. 
\begin{definition}
A triple of CM-type $r$ is  \emph{almost principal} if there exists a (possibly empty) finite set $\CN_0$ of prime
ideals $\mathfrak q$ of $K$, all lying above prime ideals of $F$ which do not split in $K/F$, such that, setting 
$\mathfrak n=  \prod\nolimits_{\mathfrak q\in \CN_0} \mathfrak q,$  we have 
$$
\Ker \lambda \subset A[\iota(\mathfrak n)].
$$
\end{definition}

\section{Local invariants}\label{section3}
Let $r$ be a generalized CM-type of rank $n$ for the CM-field $K$.   In
this section we assume that $n$ is \emph{even}. 

We first suppose that $S=\Spec k$, where $k$ is any field that is, at the same time, an $\CCO_E$-algebra. 
To a triple $(A, \iota, \lambda)$ of CM-type $r$ over $S$,  we will
attach a local invariant
\begin{equation*}
 \inv_v (A, \iota, \lambda)^{\natural} \in F^\times_v / {\rm Nm} (K^\times_v)
\end{equation*}
for every place $v$ of $F$.  In addition, 
we let
$$ \inv_v (A, \iota, \lambda)= \chi_v( \inv_v (A, \iota, \lambda)^{\natural}) =\pm 1 ,$$
where $\chi_v$ is the character of $F_v^\times$ associated to $K_v/F_v$. 
In particular, if $v$ is non-archimedean and split in $K$, then the invariant is trivial.
\smallskip

a) \emph{Archimedean places.} If $v$ is archimedean, we set 
\begin{equation*}
 \inv_v (A, \iota, \lambda) = \inv_v (A, \iota, \lambda)^{\natural} =(-1)^{r_\varphi + n(n-1)/2}.%\in F^{\times}_v / {\rm Nm} (K^\times_v).
\end{equation*}
Here $\varphi$ is either of the two complex embeddings of $K_v$.  Recall that $n$ is even. Note that the factor $(-1)^{n(n-1)/2}$ 
is included in analogy with the standard definition of the discriminant of a quadratic or hermitian form.  
There it is included so that the discriminant  depends only on the Witt class of the form, i.e., does not change 
if a hyperbolic plane is added.

\smallskip

b) \emph{Non-archimedean places not dividing ${\rm char} \, k$.} Let $v \vert \ell$. We fix a trivialization 
of the $\ell$-power roots of unity in an algebraic closure $\bar k$ of $k$, i.e. an isomorphism 
$\BZ_\ell(1)_{\bar k} \simeq \BZ_{\ell \,\bar k}$.

Let $V_\ell (A)$ be the rational $\ell$-adic Tate module of $A$. Then $V_\ell (A)$ is a free $K \otimes \BQ_\ell$-module of rank $ n$. Due to the trivialization of 
$\BZ_\ell(1)_{\bar k}$, the polarisation $\lambda$ determines an alternating bilinear form
\begin{equation}\label{Weilpairing}
 \langle \, , \rangle_\lambda \, : V_\ell(A) \times V_\ell(A) \longrightarrow \BQ_\ell \, ,
\end{equation}
satisfying
\begin{equation}
 \langle \iota (a) x,y \rangle_\lambda \, = \, \langle x, \iota(\bar a)y\rangle_\lambda \, \quad , \quad a \in \CCO_K .
\end{equation}
According to the decomposition $ F\otimes \BQ_\ell = \prod\nolimits_{w \vert \ell}  F_w$ , we can write
\begin{equation}
 V_\ell(A) =  \bigoplus\nolimits_{w \vert \ell} \, V_w(A) \, ,
\end{equation}
where $V_w(A)$ is a free $K_w$-module of rank $ n$. We can write
\begin{equation}
 \langle \, , \rangle_\lambda \, = \sum\nolimits_{w \vert \ell}  \, {\rm Tr}_{F_w/\BQ_\ell} \langle \, ,\, \rangle_{\lambda, w} \, ,
\end{equation}
where $\langle \, , \rangle_{\lambda, w}$ is an alternating $F_w$-bilinear form on $V_w(A)$. Now $K_v$ is a field by assumption. Let $\delta_v \in K_v^\times$
with $\bar \delta_v = - \delta_v$. Then we can write
\begin{equation}
 \langle x,y \rangle_{\lambda, v} \, = {\rm Tr}_{K_v/F_v} (\delta_v^{-1} \cdot (x,y)_v) \, ,
\end{equation}
for a unique $K_v/F_v$-hermitian form $(\cdot \, , \cdot)_v$ on $V_v(A)$.

We then define the local invariant at $v$ as the discriminant of the hermitian form,
\begin{equation}
 \inv_v  (A, \iota, \lambda)^\natural = \disc\, (\,,\,)_v \in F^{ \times}_v/{\rm Nm} (K^\times_v) .
\end{equation}
Recall that 
\begin{equation}
\disc\, (\,,\,)_v = (-1)^{n(n-1)/2}\,\det((x_i,x_j)_v) \in F^{ \times}_v/{\rm Nm} (K^\times_v),
\end{equation}
where $\{x_i\}$ is any $K_v$-basis for $V_v(A)$. 
This is well-defined, independent of the auxiliary choices made. Indeed, any two choices of $\delta_v$
 differ by an element in $F_v^{\times}$ and a different choice changes the
discriminant by a factor in $F_v^{\times, n} \subset {\rm Nm} (K^\times_v)$ (recall that $n$ is even).  Similarly,
a different trivialization of $\BZ_\ell(1)_{\bar k}$ leaves the discriminant unchanged. We also note that
$\inv_v (A, \iota, \lambda)$ is unchanged after any base change $k \rightarrow k'$.

Note that  $\inv_v (A, \iota, \lambda)=1$ for almost all places.  More precisely, suppose that $v\vert \ell$ where $\ell$ 
is an odd prime that is unramified in $K$ and 
$\ell $ does not divide $|\Ker(\lambda)|$.  Then the Tate module $T_\ell(A)\subset V_\ell(A)$ is a free $\CCO_K\otimes_\BZ\BZ_\ell$-module of rank $n$
and
is self dual with respect to $\langle\ ,\ \rangle_{\lambda}$.  Taking $\delta_v$ to be a unit, we obtain a unimodular hermitian lattice 
$T_\ell(A)_v$ in $V_v(A)$, and hence $\disc (\ ,\ )_v$ is a unit and hence a norm. 
\smallskip

c) \emph{Non-archimedean places dividing $p = {\rm char}$} $k$. Let $v \vert p$. Let $\bar k$ be an algebraic
closure of $k$, and let $M = M(A_{\bar k})$ be the Dieudonn\'{e} module of the abelian variety 
$A_{\bar k} = A \otimes_k \bar k$. Then $M_\BQ = M \otimes_\BZ \BQ$ is a free 
$K\otimes_\BZ W(\bar k)$-module of rank $n$, with the $\sigma$-linear Frobenius operator $\uF$. 
Under the decomposition $ F\otimes_\BQ \BQ_p =  \prod\nolimits_{w \vert p} F_w$, we obtain a decomposition
\begin{equation*}
 M_\BQ =  \bigoplus\nolimits_{w \vert p} M_{\BQ,w} \, ,
\end{equation*}
where $M_{\BQ,w}$ is a free $F_w \otimes_{\BZ_p} W(\bar k)$-module of rank $2 n$, stable under the Frobenius. 
In particular,  for our fixed $v$,  $M_{\BQ, v}$ is a free $K_v \otimes_{\BZ_p} W(\bar k)$-module of rank $n$. As in the $\ell$-adic case,
$M_{\BQ, v}$ is equipped with a hermitian form
\begin{equation*}
 ( \, , )_v : M_{\BQ,v} \times M_{\BQ,v} \lra K_v \otimes_{\BZ_p} W(\bar k) \, ,
\end{equation*}
which, in this case, satisfies the additional condition 
\begin{equation*}
 (\uF x, \uF y)_v = p \cdot (x,y)^\sigma_v \quad , \quad \forall x,y \in M_{\BQ, v} \, .
\end{equation*}
%Let $K^t_v$ be the maximal subextension of $K_v$ unramified over $\BQ_p$, and let $f = [K^t_v : \BQ_p]$. Let
%$W(\bar k)_{K_v} = K_v \otimes_{K^t_v} W(\bar k)_\BQ$.
%Then $K_v \otimes W(\bar k) = (W(\bar k)_{K_v})^f$ and, correspondingly,
%\begin{equation*}
% M_{\BQ,v} =  \bigoplus\nolimits_{i \in \BZ/f} M^i_{\BQ, v} \, ,
%\end{equation*}
%such that ${\rm deg} \, F = 1$. 
Let 
\begin{equation*}
 N_{\BQ,v} = \bigwedge^n_{K_v\otimes_{\BZ_p} W(\bar k)} \, M_{\BQ, v} \, .
\end{equation*}
Then $N_{\BQ, v}$ is a free $K_v \otimes_{\BZ_p} W(\bar k)$-module of rank $1$, equipped with a hermitian form
\begin{equation*}
 ( \, , )_v : N_{\BQ,v} \times N_{\BQ,v} \lra K_v \otimes_{\BZ_p} W(\bar k)
\end{equation*}
satisfying
\begin{equation*}
 (\uF x, \uF y)_v = p^n \cdot (x,y)_v^\sigma \, .
\end{equation*}
Furthermore, $N_{\BQ,v}$ is an isoclinic rational Dieudonn\'{e} module of slope $\frac{n}{2}$. 
 Since $n$ is
even, there exists $x_0 \in N_{\BQ, v}$ with $\uF x_0 = p^{\frac{n}{2}} \cdot x_0$, such that $x_0$ generates
the $K_v \otimes_{\BZ_p} W(\bar k)$-module $N_{\BQ, v}$. From 
\begin{equation*}
 p^n \cdot ( x_0, x_0)_v = (p^{\frac{n}{2}} x_0, p^{\frac{n}{2}} x_0)_v = (\uF x_0, \uF x_0)_v = p^n \cdot (x_0, x_0)^\sigma_v \, , 
\end{equation*}
it follows that $(x_0, x_0)_v \in F_v^\times$. The local invariant is the residue class %of $(x_0, x_0)_v$,
\begin{equation*}
\inv_v (A, \iota, \lambda)^\natural = (-1)^{n(n-1)/2}\,(x_0, x_0)_v \in F^{ \times}_v / {\rm Nm} (K^\times_v). 
\end{equation*}
It is easy to see that this definition is independent of all choices, i.e., of the algebraic closure $\bar k$ of $k$,
of the scaling $\delta_v$ of the hermitian form $( \, , )_v$, and of the choice of the generator 
$x_0$ of $N_{\BQ, v}$ above.
\begin{remark} The invariant at a $p$-adic place is analogous to Ogus's \emph{crystalline discriminant}, cf. \cite{O}.
\end{remark}
\begin{proposition}\label{locconst}
 Let $(A, \iota, \lambda) \in \CM_r (S)$, where $S$ is a connected scheme. 
Then for every place $v$ of $F$, the function
\begin{equation*}
 s \lto \inv_v (A_s, \iota_s, \lambda_s)
\end{equation*}
is constant on $S$.
\end{proposition}
\begin{proof}
 The assertion is trivial for archimedean places and for places split in $K$. It is obvious  for places over $\ell$ invertible in  $\mathcal O_S$ 
because then the $\ell$-adic Tate module $V_\ell (A_s)$ is the fiber of a lisse $\ell$-adic sheaf on $S$. The triple $(A, \iota, \lambda)$ is defined over a scheme of finite type over $\Spec\, \BZ$, hence we may assume that $S$ has this property. In addition, we may assume that $S$ is reduced. By the previous remarks, we are done if $S$ is of finite type over $\Spec\,\BQ$.

Now suppose that $p \cdot \mathcal O_S = 0$ and that $v|p$. Then we may assume that $S$ is a scheme of finite type over an algebraically
closed field $k$, and then further that $S$ is a smooth affine curve. But then we may choose a lifting $(T, F_T)$ of
$(S, F_S)$ over $W(k)$ and consider the value $H$ of the crystal of $A$ on the $PD$-embedding $S \incl T$.
Then the Dieudonn\'{e} module of the fiber $A_s$ at a point $s$ is equal to the fiber at $s$ of $H$, and the invariant 
$\inv_v (A_s, \iota_s, \lambda_s)$ depends on the value at $s$ of a section of a lisse 
$\BZ_p$-sheaf on $S$, defined by the analogous procedure as above, replacing $W(k)$ by $\mathcal O_T$. Hence it is locally constant in this case.

Finally, assume that $S$ is flat and of finite type over $\Spec\,\BZ$. By what precedes, we have to prove that $\inv_v$ remains unchanged under specialization along a DVR of unequal characteristic $(0, p)$ when $v|p$. Thus let $\CCO$ be a complete DVR with residue field $k$ of characteristic  $p$ and fraction field $L$ of characteristic $0$.  Let $(\tilde A, \tilde \iota, \tilde \lambda)$ be a CM-triple of type $r$ over $O$  lifting $(A, \iota, \lambda)$ over $k$. 
By $p$-adic Hodge theory, there is a canonical isomorphism
\begin{equation}
V_p(\tilde A_L)\otimes_{\BQ_p} B_{\rm crys}\simeq M(A)_\BQ\otimes_{W(k)_\BQ} B_{\rm crys} \, ,
\end{equation}
compatible with all structures on both sides, in particular, with the Frobenii, with the $K$-actions on both sides and with  the polarization forms, cf. \cite{Fa, T}. Here $B_{\rm crys}$ denotes Fontaine's period ring, cf. \cite{Fo}. Moreover, after extension of scalars under the inclusion $B_{\rm crys}\subset  B_{\rm dR}$, this isomorphism is compatible with the filtrations on both sides. 

Decomposing both sides with respect   to the actions of $F\otimes\BQ_p$, we obtain, for any place $v\vert p$ of $F$ that does not split in $K$, corresponding isomorphisms of 
$K_v$-modules
$$
V_v(\tilde A_L)\otimes_{\BQ_p} B_{\rm crys}\simeq M(A)_{\BQ, v}\otimes_{W(k)_\BQ} B_{\rm crys} \, ,
$$
where $V_v(\tilde A_L)$ is the summand  of $V_p(\tilde A_L)$  corresponding to $v$ in the product decomposition $ F\otimes_\BQ \BQ_p =  \prod\nolimits_{w \vert p} F_w$, and where the other notation is taken from the definition of the local invariant at $v$, given  in c) above. 

Let $ S_{v} = \bigwedge^n_{K_v} \, V_v(\tilde A_L) \, $.
Then we obtain an isomorphism between free $K_v\otimes_{\BQ_p}B_{\rm crys}$-modules of rank one, 
\begin{equation}\label{isowedge}
S_v\otimes_{\BQ_p} B_{\rm crys}\simeq N_{\BQ, v}\otimes_{W(k)_\BQ} B_{\rm crys} \, .
\end{equation}
Now we saw in the course of the definition of $\inv_v(A, \iota, \lambda)$ above, that  $N_{\BQ, v}={N_{0, v}}(\frac{n}{2})$, where $N_{0, v}$ is a multiple of the unit object in the category of filtered Dieudonn\'e modules (even as a {\it filtered} Dieudonn\'e module). Untwisting and taking  on both sides of \eqref{isowedge} the subsets in the $0$-th filtration part where the Frobenius acts trivially, we obtain an isomorphism of Galois modules
 \begin{equation*}
S_v\simeq {\bar N_{0, v}}(\frac{n}{2})\, ,
\end{equation*}
where ${\bar N_{0, v}}$ is the Galois representation corresponding to ${N_{0, v}}$, cf.~\cite{Fo}. 
By the functoriality of this isomorphism, it is compatible with the hermitian forms on these one-dimensional $K_v$-vector spaces. It therefore follows from the definition of the local invariant that 
\begin{equation}
\inv_v(A, \iota, \lambda)=\inv_v\big((\tilde A, \tilde \iota, \tilde \lambda)_L\big)\, .
\end{equation}
 \end{proof}
 \begin{remark}
The last part of the proof above is analogous to the proof in \cite{K} that the {\it Kottwitz invariant} is trivial. 
\end{remark}

 \begin{question}\label{conj.prodform}
Let $k$ be any field and consider a  CM-triple $(A, \iota, \lambda) \in \CM_r (k)$. 
When is the product formula satisfied, 
\begin{equation*}
\prod_v  \inv_v (A, \iota, \lambda)= 1 ?
\end{equation*}
\end{question}
Note that one can suppose in this question that $k$ is algebraically closed.  This question looses its sense when the generalized CM-type cannot be read off from $(A, \iota, \lambda)$. This happens for instance when $F=\BQ$ and ${\rm char}\, k=p$, where $p$ ramifies in $K$, cf. \cite{PRS}. Indeed, in this case the two $\BQ_p$-embeddings of $K$ 
into $\bar\BQ_p$ induce identical homomorphisms from $\CCO_K$ into $k$, and hence the generalised CM-type for $(A, \iota, \lambda)$ can be changed  arbitrarily, without violating the Kottwitz condition \eqref{signature.condition}.

\begin{proposition}\label{prod.formula}
 The product formula is satisfied in the following cases.

\smallskip

 \noindent (i) If ${\rm char} \, k = 0$.
 
 \smallskip
 
\noindent (ii) If ${\rm char} \, k = p > 0$, and $(A, \iota, \lambda)$ can be lifted as an object of $\CM_r$ to a DVR with 
residue field $k$ and fraction field of characteristic zero. More generally, the same is true if there exists
an $\CCO_K$-linear isogeny $\alpha : A' \rightarrow A$, such that $A'$ can be lifted to characteristic zero in the previous sense,
compatible with the isogeny action $K \rightarrow {\rm End}(A') \otimes \BQ$ and the polarization 
$\alpha^* (\lambda)$ of $A'$, provided that $A'$ is of generalized CM-type $r'$ with 
$r'_\varphi \equiv r_\varphi \, \, {\rm mod} \, 2,  \forall \varphi$.
\end{proposition}
\begin{proof} The  condition on the generalised CM-type in (ii) makes sense since in characteristic zero $K$ acts on the Lie algebra, and the CM-type of the Lie algebra can be read off from this action of $K$.

Let us first prove (i). We may assume first that $k$ is a field extension of finite type of $\BQ$, and then, by the invariance of our definitions under extension of scalars, that $k=\BC$. Then the first rational homology group $U=H_1(A, \BQ)$ is equipped with an action of $K$ and a symplectic form
$\langle\, , \, \rangle$ which, after extension of scalars to $\BQ_\ell$,  gives the $\ell$-adic Tate module of $A$.  Fix $\delta\in K$ and its associated standard CM-type $\Phi$ as in the beginning of the next section. By the same procedure 
as above, $U$ is equipped with a hermitian form $(\, ,\, )_U$,  comp. also \eqref{norm.hermit} below. Let $\inv_v(U)$ be the local invariants of this hermitian vector space, comp.~loc.~cit. They satisfy the product formula. By the compatibility with the $\ell$-adic 
Tate modules, it follows for any $v\in \Sigma_f(F)$ that $\inv_v(U)=\inv_v(A, \iota, \lambda)$. To complete the proof, we have to see that, for any archimedean place $v$, we have $\inv_v(U)=(-1)^{r_\varphi+n(n-1)/2}$, i.e., that ${\rm sig}((U, \ (\ ,\ )_U) = ((r_{\varphi}, r_{\bar\varphi}))_{\varphi\in \Phi}$. This is proved in the next section, right after \eqref{localinvU}. 

\begin{comment}
Under the isomorphism 
\begin{equation}\label{LieAfirsttime}
\Lie(A)\ \overset{\sim}{\longleftarrow} \  U\otimes_\BQ\BR \ \overset{\sim}{\lra}\  \prod_{v\in \Sigma_\infty(F)} U\otimes_{F,v}\BR,
\end{equation}
the complex structure on $\Lie(A)$ induces a complex structure $J$ on $U\otimes_\BQ\BR$ which preserves each of the factors on the right. 
But now the signature condition (\ref{signature}) (and the fact that the hermitian form on $U$ is defined in terms of the alternating form $\langle\, , \, \rangle$) implies that
$$\text{\rm sig}(U, \ (\ ,\ )_U) = ((r_{\varphi}, r_{\bar\varphi}))_{\varphi\in \Phi}\, ,$$
as asserted. 
\end{comment}

Now let us prove (ii). Let $(\tilde A, \tilde \iota, \tilde \lambda)$ be a CM-triple of type $r$ over a DVR $\CCO$ with residue field $k$ and fraction field $L$ of characteristic $0$ lifting $(A, \iota, \lambda)$. Then by Proposition \ref{locconst}, for every place $v$, 
\begin{equation}
\inv_v(A, \iota, \lambda)=\inv_v\big((\tilde A, \tilde \iota, \tilde \lambda)_L\big)\, .
\end{equation}
 Now the product formula for $(\tilde A, \tilde \iota, \tilde \lambda)_L$, valid by (i),  implies the assertion. 
 
 The addendum in (ii) follows easily by observing that the local invariants only depend on the isogeny class of the CM-triple in the sense made precise in the statement. 
\end{proof}
\begin{corollary}
Let $k$ be a field of characteristic $p>2$ such that the CM field $K$ is absolutely unramified at $p$. Then the product formula is valid for any CM-triple $(A, \iota, \lambda) \in \CM_r (k)$. 
\end{corollary}
\begin{proof} This is a consequence of G\"ortz's flatness  theorem \cite{Goertz1} which implies that any such CM-triple can be lifted to characteristic zero. 
\end{proof}
\begin{comment}
\begin{remark}Let $k$ be an algebraically closed field of characteristic $p$, and let  $(A, \iota, \lambda)$ be a CM-triple over $k$. It seems reasonable to conjecture that there exists a CM-triple $(A', \iota', \lambda')$ over $k$ and 
an $\CCO_K$-linear isogeny $\alpha : A' \rightarrow A$ with  $\alpha^* (\lambda)$  a multiple of $\lambda'$, such that $(A', \iota', \lambda')$ can be lifted to a DVR with 
residue field $k$ and fraction field of characteristic zero. By the previous proposition, this would imply the product formula conjecture. 
\end{remark}
\end{comment}
\section{Formulation of the moduli problem}

In this section we fix a CM-field $K$ with totally real subfield $F$ and an element $\delta\in K^\times$ with 
$\bar\delta= - \delta$.  This element 
determines a (standard) CM type $\Phi$ by the condition
$$\Phi = \{ \ \varphi \in \Hom_{\BQ}(K,\bar\BQ)\mid \Im(\varphi(\delta))>0\ \}.$$ 
Let $r$ be a generalized CM-type
of rank $2$ with reflex field $E=E(r)$. 
We also fix a function $\htt$ on the set of  prime ideals of $F$,
\begin{equation}
 \htt : \Sigma_f(F) \lra \{0, 1, 2\},\quad  \, {\bf p} \mapsto \htt_{\bf p} \, ,
\end{equation}
with finite support contained in the set  of  prime ideals that are non-split in $K/F$.
%We make the assumption that the value $\htt_{\bf p}$ is bounded by $2$. 
We also sometimes write $\htt_{\bf q}$ for $\htt_{\bf p}$, where ${\bf q}$ denotes
the prime ideal of $K$ over  ${\bf p}$, and we let 
$$\frak n = \frak n(\htt) = \prod_{\substack{{\bf q}\\ \htt_{\bf q} \ne 0}} {\bf q}.$$

\begin{definition}\label{stackalmostprinc} Given data $(K/F, r, \htt)$, 
let 
$\CM_{r, \htt}$ be the DM-stack over $({\rm Sch} / \CCO_E)$ with
\begin{equation*}
 \CM_{r, \htt}(S) = \text{category of  triples } \, (A, \iota, \lambda) \, \text{ of CM-type $r$ over $S$},
\end{equation*}
satisfying the following  condition,   

\smallskip

\noindent {\it The triple $(A, \iota, \lambda)$ of CM-type $r$ is almost principal with $\Ker \lambda \subset A[\iota(\frak n)]$ and}
\begin{equation}\label{ker.rank}
\vert\Ker \, \lambda \vert = \prod\nolimits_{{\bf q}} N({\bf q})^{ \htt_{\bf q}}\, .
\end{equation}
Here $N({\bf q}) = |O_K/{\bf q}|$. 
\end{definition}
This DM stack is not connected. We further decompose it as follows. Let $\big (V, (\, ,\, )_V\big )$ be a hermitian vector space of dimension $2$ over $K$ such that $V\otimes_{K,\varphi} \BC$
has signature $(r_\varphi, r_{\bar \varphi})$, for every $\varphi \in \Phi$.  Note that this is consistent with (\ref{signature}). We introduce the DM stack $\CM_{r, \htt, V}$ of triples $(A, \iota, \lambda)\in \CM_{r, \htt}$ such that 
\begin{equation}\label{inv.condition}
 \inv_v  (A, \iota, \lambda) = \inv_v (V) \, , \quad \forall v\in \Sigma(F) \, .
\end{equation}
 By Proposition \ref{locconst}, this is an open and closed substack of $\CM_{r, \htt}$. Note that the stack $\CM_{r, \htt, V}$ may be non-flat at certain places $\nu$ of $E$. Let $\nu$ lie above a prime number $p$. Then non-flatness can occur  if  $p$ ramifies in $K$, or if $p$ is divided by prime ideals of $F$ in the 
support of $\htt$.

Let $G$ be the group of unitary similitudes of $V$ with 
similitude factor in $\BQ$, i.e., the linear algebraic group over $\BQ$, with values in a $\BQ$-algebra $R$ given by
\begin{equation*}
 G(R) = \{ g \in \GL_{K \otimes R}(V \otimes_\BQ R) \mid (gv, gw) = c(g) \cdot (v,w), \ c(g) \in R^\times \} .
\end{equation*}
In particular, if we let 
$$\Phi_{a} = \{ \ \varphi \in \Phi \mid r_{\varphi} =a\ \},$$
then\footnote{
We take the hermitian form on $\BC^2$ to be $\text{\rm diag}(1,-1)$ for the $\Phi_1$ factors and $\pm 1_2$ for the other factors.} 
$$G(\mathbb R) \simeq \bigg(\ \prod_{\varphi\in \Phi_1} \text{\rm GU}(1,1)\times \prod_{\varphi\in \Phi_2\cup \Phi_0} \text{\rm GU}(2)\ \bigg)_0$$
where the subscript $0$ denotes the subgroup for which all the scale factors coincide. 
For $\BS = \text{\rm Res}_{\BC/\mathbb R}\mathbb G_m$, let
\begin{equation}\label{defofh}
\bold h: \BS \lra G_{\mathbb R}, \qquad \bold h: z\longmapsto (\bold h_{\varphi}(z))_{\varphi\in \Phi},
\end{equation}
where
$$\bold h_{\varphi}(z) = \begin{cases}
\begin{pmatrix} z&{}\\{}&\bar z\end{pmatrix} &\text{ for $\varphi \in \Phi_1$,}\\
\nass
z\cdot 1_2 &\text{ for $\varphi \in \Phi_2$,}\\
\nass
\bar z\cdot 1_2 &\text{ for $\varphi \in \Phi_0$.}\\
\end{cases}
$$
Note that $c\circ \bold  h(z) = |z|^2$.  Also note that the image of the scale map is
$$c(G(\mathbb R)) =\begin{cases} \mathbb R^\times &\text{ if $\Phi_2 \cup \Phi_0$ is empty,  and}\\
\nass
\mathbb R^\times_{>0}&\text{ otherwise.}
\end{cases} 
$$  
We let $G(\BR)^+$ be the subgroup of $G(\BR)$ for which the scale is positive.
\begin{proposition}\label{prop.nonempty}

  Assume that all finite places $v$ of $F$ with $v\vert 2$ are unramified in $K/F$.   

\smallskip

\noindent (i) The set $\CM_{r, \htt, V}(\BC)$ of complex points of the
moduli space $\CM_{r, \htt, V}$ is non-empty 
if and only if  the following compatibility conditions between the function $\htt$ and the invariants of $V$ are satisfied.
\begin{enumerate}
\item  If $\htt_{\bold p_v}=0$ and $v$ is inert in $K/F$, then $\inv_v(V)=1$. 
\item  If $\htt_{\bold p_v}=2$, then $\inv_v(V)=1$. 
\item If $\htt_{\bold p_v}=1$, then $v$ is inert in $K/F$ and $\inv_v(V)=-1$. 
\end{enumerate}

\smallskip

\noindent (ii) In this case 
$\CM_{r, \htt, V} \otimes_{\CCO_E} \BC$ is the Shimura variety ${\rm Sh}_C(G, \bold X)$ associated to the pair $(G, \bold X)$, where $\bold X$ is the $G(\BR)$-conjugacy class of 
homomorphisms $\bold h: \BS\to G_\BR$ given by (\ref{defofh}),  and where the compact open subgroup $C$ of $G(\BA^\infty)$ is the stabilizer of an $O_K$-lattice
in $V$ satisfying conditions (\ref{AP1}) and (\ref{AP2})  in Lemma~\ref{lattice.lemma} below. 
%It is connected (if non-empty).
\end{proposition}
\begin{proof}
Let $(A,\iota,\lambda)$ be a CM-triple of type $r$ over $\BC$.  Then $U= H_1(A,\BQ)$ is a $2$-dimensional $K$-vector space, and the Riemann form 
determined by $\lambda$ is an alternating, $\BQ$-bilinear form $\langle \ ,\ \rangle:U\times U \rightarrow \BQ$ such that 
$\langle \iota(a)x,y\rangle \, =\,  \langle x,\iota(\bar a)y\rangle $, for all $a\in K$. There is then a unique $K$-valued hermitian form $(\ ,\ )_U$ on $U$ such that   
\begin{equation}\label{norm.hermit}
 \langle x,y \rangle\  = \ \tr_{K/\BQ}(\delta^{-1}(x,y)_U),
\end{equation}
where $\delta = - \bar \delta \in K^\times$ is the element fixed at the beginning of this section. For each place $v$ of $F$, 
the hermitian space $(U, \ (\ ,\ )_U)$ has invariant
$$\inv_v(U, \ (\ ,\ )_U)=\chi_v(-\det((u_i,u_j)_U)) \in \{\pm 1\},$$
where $\{u_1, u_2\}$ is a $K$-basis for $U$ and $\chi_v(x)= (\delta^2,x)_v$ is the quadratic character associated to $K_v/F_v$. Here $(a, b)_v$ is the quadratic Hilbert symbol for $F_v$. 
By the compatibility of the $\BQ_\ell$-bilinear extension of the Riemann form and the form (\ref{Weilpairing}) arising from the Weil pairing
and by (\ref{inv.condition}), we have 
\begin{equation}\label{localinvU}
\inv_v(U, \ (\ ,\ )_U) = \inv_v(A,\iota,\lambda) = \inv_v(V, \ (\ ,\ )_V),
\end{equation}
for all finite places $v$.  On the other hand, under the isomorphism 
\begin{equation}\label{LieA}
\Lie(A)\ \overset{\sim}{\longleftarrow} \  U\otimes_\BQ\BR \ \overset{\sim}{\lra}\  \prod_{v\in \Sigma_\infty(F)} U\otimes_{F,v}\BR,
\end{equation}
the complex structure on $\Lie(A)$ induces a complex structure $J$ on $U\otimes_\BQ\BR$ which preserves each of the factors on the right. 
On the $v$-th factor on the right side of (\ref{LieA}), there is a complex structure 
$$J_\delta = \iota(\delta)\otimes N_{K/F}(\delta)_v^{-\frac12}.$$
Recall that $N_{K/F}(\delta)$ is totally positive. 
The signature condition (\ref{signature}) implies that, for $\varphi\in \Phi$,   $J_\delta = J$ when $r_{\varphi} =2$, $J_\delta = - J$ when $r_{\varphi}=0$, 
and $J_\delta$ has eigenvalues $\pm i$ on $(U\otimes_{F,v}\BR, J)$ when $r_{\varphi}=1$.  From this it follows that the signature
$$\text{\rm sig}(U, \ (\ ,\ )_U) = ((r_{\varphi}, r_{\bar\varphi}))_{\varphi\in \Phi}$$
coincides with that of $V$.  Together with (\ref{localinvU}), this implies that there is an isometry
$\eta:U\overset{\sim}{\lra} V$ of hermitian spaces over $K$.  Via $\eta$ and (\ref{LieA}), the action of $\BC^\times$ on $\Lie(A)$ yields a 
homomorphism $\bold h_A: \mathbb S\lra G_\BR$. It is conjugate by $G(\BR)^+$ to the homomorphism $\bold h$ defined by (\ref{defofh}).

Next, we must take into account the almost principal condition (i) in Definition \ref{stackalmostprinc}, which amounts to the following.
For convenience, we write $(\ ,\ )_V$ of the hermitian form on $V$.
\begin{lemma}\label{lattice.lemma} (i)
Let $L= H_1(A,\BZ) \subset U$ with dual 
lattice $L^\vee$ with respect to $\langle \ ,\ \rangle$.  
Let $M=\eta(L)$, resp.  $M^\vee=\eta(L^\vee)$, be the image of $L$, resp. $L^\vee$,  under $\eta$. 
Then $M$ and $M^\vee$ are  $O_K$-lattices in $V$ satisfying the conditions
\begin{equation}\label{AP1}
M\subset M^\vee \subset \frak n^{-1}M,
\end{equation}
and 
\begin{equation}\label{AP2}
|M^\vee/M| = \prod\nolimits_{{\bf q}} N({\bf q})^{ \htt_{\bf q}}\, .
\end{equation} 
Furthermore, 
\begin{equation*}
M^\vee= \{ \ x\in V\mid\  \delta^{-1}( x, M )_V \subset \partial_K^{-1}\ \},
\end{equation*}
with $\partial_K$ the different  of $K/\BQ$.

The subgroup 
$$G(\BA^\infty)^0 = \{ \ g\in G(\BA^\infty) \mid c(g)\in \widehat{\BZ}^\times\ \}$$ 
of $G(\BA^\infty)$ acts on the set of such lattices. 

\smallskip

\noindent (ii)  Assume that all finite places $v$ of $F$ with $v\vert 2$ are unramified in $K/F$.  
Then, for any finite place $v$ of $F$, a lattice in $V_v$ satisfying the local analogues of (\ref{AP1}) and (\ref{AP2})
is unique up to isometry. In particular, the isometry group $U(V)(F_v)$ acts transitively 
on the set of such lattices.
\end{lemma}

 \begin{proof}
For  a finite place $v\in \Sigma_f(F)$ of $F$,  choose $\delta_v\in K_v^\times$ such that $\bar\delta_v=-\delta_v$ and $\delta_v O_{K_v} = \partial_{K,v}$.  
Note that, for almost all places, we can take $\delta_v=\delta$. 
Then the lattice $M_v$ in $V_v$ has dual lattice
$$(M^\vee)_v = \{ \ x\in V\mid\  \delta_v\delta^{-1}( x, M_v )_V\subset O_{K,v}\ \},$$
with respect to the hermitian form $h_v:= \delta_v\delta^{-1}( \ , \  )_V$.  Note that the $2$-dimensional hermitian spaces $(V_v,h_v)$ and 
$(V_v,(\ ,\ )_V)$ are related to each other by scaling and hence are isometric.  Now apply the local theory of hermitian lattices. 

If $\htt_{\bold p_v}=0$, then $\mathfrak n_v=O_{K,v}$ so that $M^\vee_v=M_v$, and $M_v$ is unimodular with respect to $h_v$. 
 If $v$ is split in $K/F$, there is a unique $2$-dimensional hermitian space $V_v$, and a
self dual lattice $M_v$ in it is unique up to isometry.  
If $v$ is inert in $K/F$, then the space $(V_v,h_v)$ is split and the lattice $M_v$ is unique up to isometry, \cite{jacobowitz}, Theorem~7.1. 
If $v$ is ramified in $K/F$ (and hence non-dyadic by our assumption), the isometry class of $M_v$ is determined by $\det(M_v)\in F^\times_v/{\rm Nm}(K^\times_v)$,  
\cite{jacobowitz}, Proposition~8.1 (a).  But since the class of $V_v$ is already fixed, it follows that there is a unique isometry class of self-dual $M_v$'s in $V_v$. 

Next suppose that $\htt_{\bold p_v}=2$, so that $M_v^\vee = \pi_{K_v}^{-1} M_v$ and $M_v$ is $\pi_{K_v}$-modular. 
Note that, by our assumption about the support of the function
 $\htt$, 
the case where $v$ is split in $K/F$ is excluded.   If $v$ is inert in $K/F$, then the lattice $M_v$ is unique up to isometry, \cite{jacobowitz}, Theorem~7.1,
and the space $V_v$ is split. If $v$ is ramified in $K/F$, then the existence of a $\pi_{K_v}$-modular lattice implies that $V_v$ is split and, again, the 
lattice $M_v$ is unique up to isometry, \cite{jacobowitz}, Proposition~8.1 (b). 

Finally, suppose that $h_{\bold p_v}=1$. In particular, $M_v$ is not modular and has a Jordan decomposition of type $(1)\oplus (\pi_{K_v})$.
If $v$ is inert in $K/F$, it  follows that $V_v$ is anisotropic and that $M_v$ is unique up to isometry, \cite{jacobowitz}, Theorem~7.1. 
On the other hand, if $v$ is ramified in $K/F$, and hence non-dyadic,  there are no $\pi_{K_v}$-modular lattices of rank $1$, \cite{jacobowitz}, Proposition~8.1. 
so that the condition $h_{\bold p_v}=1$ cannot occur for ramified places $v$. 
\end{proof}

The previous lemma and its proof imply that the conditions (1)--(3) in Proposition \ref{prop.nonempty}, (i),  are equivalent to the existence of  an $O_K$-lattice $M$ in $V$ satisfying (\ref{AP1}) and (\ref{AP2}). We fix such a lattice  and let $C\subset G(\BA^\infty)$ be its stabilizer. 
Note that, by the lemma, $c(C) = c(G(\BA^\infty)^0)$ and the set of all lattices in $V$ satisfying (\ref{AP1}) and (\ref{AP2})
can be identified with $G(\BA^\infty)^0/C$.  Let $\bold X^+$ be the $G(\BR)^+$ conjugacy class of $\bold h$.  
The pair $(\bold h_A, M)\in \bold X^+\times G(\BA^\infty)^0/C$ depends on the choice of isometry $\eta$. Removing this 
dependence, we have a map
\begin{equation}\label{complex.points}
\CM_{r, \htt, V}( \BC) \lra U(V)(\BQ) \backslash \bold X^+ \times G(\BA^\infty)^0/C \overset{\sim}{\lra} G(\BQ)\backslash \bold X \times G(\BA^\infty)/C.
\end{equation}
Here note that $c(G(\BA^\infty)^0) = \widehat{\BZ}^\times$ and that 
$c(G(\BQ)\cap G(\BA^\infty)^0) = 1$ if $c(G(\BR))= \BR^\times_+$ and $\pm1$ if $c(G(\BR))= \BR^\times$.
It is easily checked that (\ref{complex.points}) is surjective and induces a 
bijection on the set of isomorphism classes on the left hand side.\end{proof}
\begin{remark}
 $\CM_{r, \htt, V}\otimes_{\CCO_E}E$ is the canonical model in the sense of Deligne of the Shimura variety ${\rm Sh}_C(G, \bold X)$, but we will not stop to show this here. 
\end{remark}
\begin{remark} \label{p-variantofmoduli}
We will also use the following variant of $\CM_{r, \htt, V}$. 
We fix a prime number $p$ and an $O_K$-lattice $M$ in $V$ satisfying (\ref{AP1}) and (\ref{AP2}), and let $C_M^p\subset G(\BA^{\infty, p})$ be the stabilizer of $M\otimes\widehat{\BZ}^p$ in $G(\BA^{\infty, p})$. 
Let $C^p \subset G(\BA^{\infty, p})$ be an open compact subgroup which is a subgroup of finite index in $C_M^p$. Let $\CCO_{E_{(p)}}$ be the localization of $\CCO_E$ at $p$. The variant 
$\CM_{r, \htt, V}( C^p)$ is the stack over $(Sch/\CCO_{E_{(p)}})$ which, in 
addition to $(A, \iota, \lambda)$ satisfying conditions (\ref{ker.rank}) and (\ref{inv.condition}), fixes a level structure ${\rm mod} \, C^p$, i.e. an isomorphism compatible with $\iota$ and with the alternating forms on both sides up to a unit in $\widehat{\BZ}^p$,
\begin{equation*}
 \wh T^p (A)  \simeq M\otimes\widehat{\BZ}^p \, {\rm mod} \, C^p \, ,
\end{equation*}
in the sense of Kottwitz \cite{K}. Here on the RHS, we use the alternating form 
\begin{equation}\label{altonV}
 \langle x,y \rangle_V  = \ \tr_{K/\BQ}(\delta^{-1}(x,y)_V)\, .
\end{equation}

Note that, due to the existence of the level structure, we only have to require the condition \eqref{inv.condition} for the places $v$ over $p$ --- for all other places it is automatic. If $C^p=C_M^p$, then by Lemma \ref{lattice.lemma}, $\CM_{r, \htt, V}( C^p)=\CM_{r, \htt, V}\otimes_{\CCO_E}\CCO_{E_{(p)}}$. 
\end{remark}

\section{Uniformizing primes}
In this section we consider the local situation. We fix a prime number $p$ and an algebraic closure $\bar \BQ_p$
of $\BQ_p$.  Let $F$  be a finite extension of $\BQ_p$ with $|F:\BQ_p|=d$,  and  let $K/F$ be an 
\'{e}tale algebra of rank $2$. We begin with the obvious local analogues of the definitions of section 2. 

A generalized CM-type $r$ of rank $n$ relative to $K/F$ is a function 
$$r: \Hom_{\BQ_p}(K, \bar\BQ_p)\lra \BZ_{\geqslant 0}, \qquad \,\, \varphi \mapsto r_\varphi,$$
such that $r_{\varphi}+r_{\bar\varphi} =n$ for all $\varphi$.  Here $\bar\varphi(a) = \varphi(\bar a)$ 
where $a\mapsto \bar a$ is the non-trivial automorphism of $K$ over $F$. 
The corresponding reflex field $E=E(r)$ is the subfield of $\bar\BQ_p$ fixed by 
$$\Gal(\bar\BQ_p/E):=\{ \tau\in \Gal(\bar\BQ_p/\BQ_p)\mid r_{\tau\varphi} = r_\varphi, \ \forall \varphi\}.$$
Let  $\CCO_E$ be the ring of integers of $E$ and let $\pi_E$ be a uniformizer of $E$. 

%The local version of a triple $(A, \iota, \lambda)$ of 
\begin{definition} A {\it triple of CM-type} $r$ over an $\CCO_E$-scheme $S$ is a triple $(X, \iota, \lambda)$, where
$X$ is a $p$-divisible group over $S$ of height $2nd$ and dimension $nd$,
$\iota : \CCO_K \rightarrow \End (X)$ is an action of the ring of integers of $K$ on $X$ satisfying the Kottwitz 
condition relative to $r$, and  $\lambda: X\rightarrow X^\vee$ is a quasi-polarization with Rosati involution inducing the non-trivial automorphism on $K/F$.
\end{definition}

\begin{definition} Such a triple is called {\it almost principal} if either $K=F\oplus F$ and $\lambda$ is principal,  or $K$ is a field
 and $\Ker \lambda$ is contained in  $X[\iota(\pi)]$, 
where $\pi=\pi_K$ denotes a uniformizer of $K$.  
\end{definition}
In particular, when $K$ is a field, $\Ker \lambda$ is a module over $\CCO_K/\pi\CCO_K$. We write the height of $\Ker\, \lambda$ in the 
form $fh$, where $f = [K^t : \BQ_p]$ is the degree of the maximal unramified subfield $K^t$ of $K$, so that $0\le h\le n$.  If $K= F\oplus F$, then $h=0$.

Now let $n$ be even.

Suppose that $k$ is an algebraically closed field of characteristic $p$ that is an $\CCO_E$-algebra.  Then, for a CM-triple $(X,\iota,\lambda)$ of type $r$
over $k$, the construction of 
c) of section 3, applied to the Dieudonn\'e module of $(X, \iota, \lambda)$,  yields an invariant
$$\inv_v(X,\iota,\lambda)^\natural \in F^\times/{\rm Nm}(K^\times),$$
and a sign 
$$\ep=\inv_v(X,\iota,\lambda)  = \chi(\inv_v(X,\iota,\lambda)^\natural)=\pm1,$$
where $\chi$ is the quadratic character attached to $K/F$.

For the rest of this section, we assume that $n=2$. 
For our description of $p$-adic uniformization,  it will be important to know when an 
almost principal CM-triple $(X, \iota, \lambda)$ 
over $k$ with given invariants $h$ and $\ep$ is unique up to isogeny.

%Let $(X, \iota, \lambda)$ be an almost principal CM-triple. 
%We write the height of $\Ker\, \lambda$ in the form $fh$, where $f = [K^t : \BQ_p]$ is the degree of the maximal unramified subfield of $K$.  
%Then  
%$0 \leq h \leq  2,$ and $h=0$ when $K= F\oplus F$. 
%When $K=F\oplus F$, we take $f= 2 f_o$ where $f_o=|F^t:\BQ_p|$. 

\begin{definition}  \label{defunifprimes}
 \noindent (i) We call $(K / F, r, h, \ep)$  \emph{uniformizing data of the first kind}, if 
 $F  = \BQ_p$, $K$ is a field,  $r_{\varphi} = r_{\bar\varphi} = 1$,
\begin{equation*} \begin{aligned}
h & = \begin{cases}
     0, \quad \text{if $\quad K/F$ is ramified}\\
1, \quad \text{if $\quad K/F$ is unramified,}
    \end{cases}
 \end{aligned}
\end{equation*}
and $\ep=-1$. 

\smallskip

 \noindent (ii) We call $(K/F, r, h, \ep)$ \emph{uniformizing data of the second kind}, if $K/F$ is an unramified field 
extension  and $r$ is of the following form: there exists a half-system $\Phi^t$ of elements
of $\Hom_{\BQ_p} (K^t, \bar \BQ_p)$ such that
\begin{equation*}
r_\varphi =\begin{cases} 
0, \quad \text{if} \quad \varphi \vert K^t \in \Phi^t\\
2, \quad \text{if} \quad \varphi \vert K^t \notin {\Phi}^t.
          \end{cases}
\end{equation*}
Note that $[K^t : F^t] = 2$.  If $h = 0$, then $\ep=1$; if  $h = 1$, then $\ep=-1$.

\smallskip

 \noindent (iii) We call $(K/F, r, h,\ep)$ \emph{uniformizing data of the third kind}, if $K = F \oplus F$ so $\ep=1$, $h = 0$,  and
\begin{equation*}
r_\varphi =\begin{cases} 
0, \quad \text{if } \varphi  \text{ factors through the first summand of $K = F \oplus F$}\\
2, \quad \text{if } \varphi \text{ factors through the second summand of $K = F \oplus F$}.
          \end{cases}
\end{equation*}
\end{definition}

\begin{proposition}\label{uniformity}
 Fix uniformizing data $(K/F, r, h, \varepsilon)$ of the first, second, or third kind. Let $(X, \iota, \lambda)$ and
$(X', \iota', \lambda')$ be two almost principal CM-triples of type $(K/F, r, h, \varepsilon)$ over $k$. %\hfill\break 
Then both $X$ and $X'$ are isoclinic $p$-divisible groups and there exists an $\CCO_K$-linear isogeny $\alpha : X \rightarrow X'$ such that
$\alpha^* (\lambda') = c\,\lambda$ with $c\in \BZ_p^\times$. \hfill \break
%Furthermore, both CM-triples $(X, \iota, \lambda)$ and $(X', \iota', \lambda')$ are \emph{ basic}, i.e., both 
%$[b]_{(X, \iota, \lambda)}$ and $[b]_{(X', \iota', \lambda')}$ are the unique basic element in $B(G, \mu_r)$. 
\end{proposition}
\begin{proof} It suffices to prove that $X$ and $X'$ are isoclinic and to show 
the existence of the isogeny $\alpha$ with $\alpha^* (\lambda) = c\lambda'$  with $c\in\BQ_p^\times$. That $c\in\BZ_p^\times$ then follows from the fact that both CM-triples are almost principal with identical $h$.

We consider the three cases separately.  

Suppose that $(X,\iota,\lambda)$ is of the first kind. Let 
$M= M(X)$ be its covariant Dieudonn\'e module.   Then $M$ is a free $W(k)$-module of rank $4$. Consider the slope decomposition of the corresponding rational Dieudonn\'e module $N=M_\BQ$. Each summand $N_\lambda$ is stable under the action of $K$ on $N$, hence has even dimension; we have $\lambda\geq 0$ if $N_\lambda\neq (0)$; furthermore, due to the polarization, $m_\lambda=m_{1-\lambda}$ for the multiplicities of the corresponding slope subspaces $N_\lambda$, resp. $N_{1-\lambda}$; finally, writing $\dim N_\lambda=m_\lambda d_\lambda$, where $d_\lambda$ is the dimension of the simple isocrystal of slope $\lambda$, we have $4=\sum_\lambda m_\lambda d_\lambda$. It follows that either $N=N_{\frac{1}{2}}$ and $m_{\frac{1}{2}}=2$, or
$N=N_0\oplus N_1$, where both summands $N_0$ and $N_1$ have dimension $2$ and $m_0=m_1=2$. The first case means that $X$ is isoclinic, the second that $X$ is ordinary. 

We need to exclude the ordinary case.  Write $W(k)_\BQ=\BQ\otimes W(k)$. The isocrystal comes with its natural polarization pairing
\begin{equation*}
\langle\, ,\, \rangle_0:N\times N\lra \, W(k)_\BQ,
\end{equation*} for which $N_0$ and $N_1$ are isotropic, and which corresponds to  a non-degenerate pairing
\begin{equation*}
[\, ,\, ]_0:N_0\times N_1\lra W(k)_\BQ.
\end{equation*} 
We choose bases $e_0, e_1$ for $N_0$, and $f_0, f_1$ for $N_1$ such that
\begin{equation}
\begin{aligned}
\underline{V}e_i=e_i, \underline{F}e_i=pe_i;\, \underline{V}f_i=pf_i, \underline{F}f_i=f_i \text{ for } i=0,1;\\ [e_0,f_0]_0=[e_1, f_1]_0=1,\,  [e_0, f_1]_0=[e_1, f_0]_0=0. 
\end{aligned}
\end{equation}
In this case, $\End\,(N)={\rm M}_2(\BQ_p)\times {\rm M}_2(\BQ_p)$. For $(b_0, b_1)\in {\rm M}_2(\BQ_p)\times {\rm M}_2(\BQ_p)$, we have 
\begin{equation*}
\langle (b_0, b_1)x, y\rangle_0=\langle x, (^tb_1, ^tb_0)y\rangle_0\, , \forall x,y\in N .
\end{equation*}
We may furthermore suppose that the action of $K$ on $N$ is given as follows. Let $K=\BQ_p(\sqrt\Delta)$. Putting $\delta=\sqrt\Delta$, 
the action of $K$ on $N_0\oplus N_1$ is given as
\begin{equation}
a+b\delta\mapsto \Big(\begin{pmatrix} a&b\\ \Delta b&a\end{pmatrix}, \begin{pmatrix} a&\Delta b\\b&a\end{pmatrix}\Big) .
\end{equation}
The given polarization on $X$ induces the pairing $\langle\, ,\, \rangle:N\times N\lra \, W(k)_\BQ$. Comparing the  involutions on $K$ induced by the Rosati involutions of the two polarizations, we see that  $\langle x , y \rangle=\langle\beta x ,y \rangle_0$, where $\beta\in \End(N)$ anticommutes with $K$.  Hence $\beta$ is of the form $\beta=(\beta_0, \beta_1)$ with 
\begin{equation}
\beta_0=\begin{pmatrix} a&b\\ -\Delta b&-a\end{pmatrix},\,\, \beta_1=\begin{pmatrix} a&-\Delta b\\  b&-a\end{pmatrix}.  
\end{equation}
Let us now calculate the local invariant according to the recipe of section \ref{section3}. The hermitian form associated to the alternating form 
$\langle\, ,\, \rangle$ is given as 
\begin{equation}\label{hermit.form}
(x, y)=\frac{1}{2}(\langle \delta x, y\rangle+\delta\langle x, y\rangle) =\frac{1}{2}(\langle \beta\delta x, y\rangle_0+\delta\langle \beta x, y\rangle_0 ). 
\end{equation} 
Taking $e_0, f_0$ as $K\otimes_{\BZ_p} W(k)$-basis of $N$, we obtain $x_0=e_0\wedge f_0$ as free generator of $\bigwedge^2N$, with $\underline{F}x_0=px_0$. Hence $\inv(X, \iota, \lambda)^\natural =-(x_0, x_0)\in \BQ_p^\times/{\rm Nm} (K^\times)$. Now
\begin{equation*}
(x_0, x_0)=\det \begin{pmatrix} (e_0, e_0)&(e_0, f_0)\\ (f_0, e_0)&(f_0, f_0)\end{pmatrix}=\det \begin{pmatrix} 0&(e_0, f_0)\\ (f_0, e_0)&0\end{pmatrix} .
\end{equation*}
\begin{comment}
From \eqref{hermit.form} we calculate $(e_0, f_0)=\Delta b+a\delta$; hence we obtain 
\begin{equation*}
(x_0, x_0)=-{\rm Nm}(\Delta b+a\delta)=\Delta {\rm Nm}(a+b\delta) .
\end{equation*}
Hence 
\begin{equation*}
\inv(X, \iota, \lambda)=\chi\big(\inv(X, \iota, \lambda)^\natural\big)=\chi\big(-\Delta {\rm Nm}(a+b\delta)\big)=\big(\Delta, -\Delta {\rm Nm}(a+b\delta)\big)=1 .
\end{equation*} 
\end{comment}
It follows that $-(x_0, x_0)\in {\rm Nm}_{K/\BQ_p}(K^\times)$. 
This contradicts the imposed sign $\varepsilon=-1$ in the definition of uniformizing data of the first kind. 

Applying the same reasoning to $(X', \iota', \lambda')$, we obtain an isogeny $\alpha: X\lra X'$ compatible with the actions of $K$. However, $\lambda=\alpha^*(\lambda') \beta$ with $\beta\in \End(N)$ invariant under the Rosati involution of $\lambda$. By precomposing $\alpha$ with $\gamma\in \End_K(N)$,  we change $\beta$ into 
$\beta\gamma \gamma^*$. We consider the solutions of the equation $c\beta\gamma \gamma^*=1$,  with $c\in \BQ_p^\times$ and $\gamma\in \End_K(N)$,  as a torsor under the $\BQ_p$-group of automorphisms of $N$ which commute with $\underline{F}$ and preserve the polarization form up to a constant.  Since this group has trivial first cohomology set, we may solve this equation, and  may change $\alpha$ so that $\alpha^*(\lambda')=c\lambda$. This finishes the case of uniformizing data of the first kind.

 If $(K/F, r, h)$ is a uniformizing data of the  third kind, the assertion is proved in \cite{RZ}, Lemma 6.41. If $(K/F, r, h)$ is a uniformizing data of the  second kind,  then again the assertion is proved in \cite{RZ}, Lemma 6.41, provided that $h=0$. In this case, the sign factor $\ep$ equals $1$, since we are then calculating the discriminant of a hermitian vector space relative to an unramified quadratic extension which admits a self-dual lattice. An inspection of the proof of loc.~cit. shows that the proof also applies to the case when $(K/F, r, h)$ is a uniformizing data of the  second kind and $h=1$, in which case $\ep=-1$. 
\end{proof}
\begin{remark}\label{travdeZink}
 As mentioned at the end of the introduction, it seems plausible that the notion of  uniformizing data of the first kind can  be generalized to include cases where  $F$ is a non-trivial extension of $\BQ_p$.  More precisely, fix an embedding $\varphi_0:F\to \bar\BQ_p$. Then the conditions on $(K/F, r, h, \varepsilon)$ become: $K$ should be a field extension of $F$, $r_\varphi$ should be equal to $1$ if $\varphi|F=\varphi_0$, and should be equal to $0$ or $2$ otherwise. Finally, as before, $h$ should be $0$ or $1$  depending on whether $K/F$ is ramified or unramified,  and $\varepsilon$ should be $-1$. 

\end{remark}

\section{Integral uniformization}\label{Integral uniformization}
In this section,  we obtain  integral $p$-adic uniformization under a whole set of assumptions that we now explain. We
fix an embedding $\nu: \bar \BQ \lra \bar \BQ_p$. This embedding also determines a $p$-adic place $\nu$ of the reflex
field $E$. We decompose $ \Hom (K, \bar \BQ)$ into a disjoint sum according to the prime ideals $\bf p$ of $F$ over $p$, 
\begin{equation*}
 \Hom (K, \bar \BQ)_{\bf p} = \{\varphi\in \Hom (K, \bar \BQ)\mid \nu\circ\varphi_{|F} \text{ induces }{\bf p}\}. 
\end{equation*}
Then  $\Hom (K, \bar \BQ)_{\bf p}=\Hom_{\BQ_p}(K\otimes_F F_{\bf p}, \bar{\BQ}_p)$. 
Let $r_{\bf p} = r \vert \Hom (K, \bar \BQ)_{\bf p}$ and $\ep_{\bold p}=\inv_{\bold p}(V)$. 

We make the assumption that $(K_{\bf p}/F_{\bf p}, r_{\bf p}, h_{\bf p}, \ep_{\bold p})$ are uniformizing data of type $1,2,$ or $3$, for all ${\bf p} \vert p$.
We note that $E_\nu$ is the composite of the local reflex fields $E({r_{\bf p}})$ (with ${\bf p}$ running over the prime ideals
of $F$ over $p$). Let $\kappa_\nu$ be the residue field of $\CCO_{E_\nu}$, and denote by $\bar \kappa_\nu$ its algebraic closure.

The proof of the integral uniformization theorem will be analogous to the proof of Theorem 6.30 in \cite{RZ}. We will proceed according to the following steps. First we will show that all $p$-divisible groups $(X, \iota,\lambda)$ which arise from points $(A, \iota, \lambda)$ of $\CM_{h,r, V} (\bar \kappa_\nu)$ are isogenous to each other, and are moreover {\it basic} in the sense of \cite{K2}. Then we prove that, in fact, all points $(A, \iota, \lambda)$ are isogenous to each other. This already yields an abstract integral $p$-adic uniformization theorem, as in \cite{RZ}. In a third step, we make this abstract uniformization theorem explicit, by making use of  the alternative moduli description of the Drinfeld halfplane in \cite{KR}.

Let $(A_0, \iota_0, \lambda_0) \in \CM_{h,r, V} (\bar \kappa_\nu)$. Let $N$ be the isocrystal of $A_0$, with its
action by $K \otimes \BQ_p$ induced by $\iota_0$, and its anti-symmetric polarization form induced by $\lambda_0$.
\begin{lemma}\label{trivial}
There is an isomorphism of $K \otimes  W (\bar \kappa_\nu)$-modules
\begin{equation*}
 N \simeq V \otimes W(\bar \kappa_\nu)\ ,
\end{equation*}
which respects the anti-symmetric bilinear forms on both sides.
\end{lemma}
\begin{proof} We have an orthogonal decomposition with respect to the anti-symmetric form \eqref{altonV},
\begin{equation}\label{orthdec}
 V \otimes \BQ_p = \bigoplus\nolimits_{{\bf p} \mid p} V_{\bf p} \, .
\end{equation}
There is a similar decomposition of $N$, orthogonal for the polarization form. Now $N_{\bold p}$ contains a parahoric lattice of type $h_{\bold p}$, i.e., a lattice $\Lambda$ such that $\Lambda\subset\Lambda^\vee\subset\pi_{\bold p}^{-1}\Lambda$ where the dimension of $\Lambda^\vee/\Lambda$ over the residue field is equal to $h_{\bold p}$; this lattice is isomorphic to the extension of scalars of the parahoric lattice of type $h_{\bold p}$ in $V_{\bold p}$, cf \cite{RZ}, Theorem 3.16, comp. \cite{RZ}, 6.12 (we use the fact that $W(\bar \kappa_\nu)$ has no non-trivial \'etale coverings). Hence we get a fortiori the isomorphism of $K \otimes  W (\bar \kappa_\nu)$-modules, as claimed.

 Another way of obtaining this isomorphism  is to note that for a complete discretely valued field with algebraically closed residue field, and a quadratic algebra over it, there is, up to isomorphism, exactly one hermitian space of given dimension. \end{proof}

Using the isomorphism of Lemma \ref{trivial}, we can write the Frobenius operator on the left hand side as $b \otimes \sigma$,
for a uniquely defined element $b \in G(W(\bar \kappa_\nu)_\BQ)$. We have $c(b) = p$, where $c : G \lra \BG_m$ denotes the
multiplier morphism. Recall Kottwitz's set $B(G)$ of $\sigma$-conjugacy classes of elements in $G(W_\BQ(\bar\kappa_\nu))$, cf. \cite{K2}. 
\begin{lemma}\label{uniquenessin prod}
 The element $[b] \in B(G)$ is basic, and independent of $(A_0, \iota_0, \lambda_0)$.
\end{lemma}
\begin{proof} Corresponding to \eqref{orthdec} there is an embedding of algebraic groups over $\BQ_p$, 
\begin{equation}\label{embed}
 G_{\BQ_p} \lra \prod\nolimits_{\bf p} G'_{{\bf p}}  \, ,
\end{equation}
where the product runs over all primes of $F$ over $p$, and where $G'_{\bf p}$ denotes the group of unitary similitudes of the hermitian space $V_{\bf p}$ 
(with similitude factor in $\BQ_p$).  Since the center of $G$ is equal to the intersection of the center of  $\prod\nolimits_{\bf p} G'_{{\bf p}}$ with $G$, an element in $B(G)$ is basic if its image under the map $B(G)\to B(\prod\nolimits_{\bf p} G'_{{\bf p}})$ is. Furthermore,  using the long exact cohomology sequence associated to the injection \eqref{embed}, this last map is injective, since the map $ \prod\nolimits_{\bf p} G'_{{\bf p}} (\BQ_p)\to (\prod\nolimits_{\bf p}\BQ_p^\times)/\BQ_p^\times$ is surjective. Hence Proposition \ref{uniformity} implies that $[b]$ is basic and independent of $(A_0, \iota_0, \lambda_0)$. 
\end{proof}
\begin{remark}\label{compkottw} In the discussion above, we have adopted the point of view of Kottwitz, that is, we  view $N$ with its additional structure as given by an element in $\tilde G(W(\bar \kappa_\nu)_\BQ)$, where $\tilde G$ is a suitable reductive algebraic group over $\BQ_p$. In fact, we have taken  $\tilde G$ to be the localization at $p$ of our group $G$ over $\BQ$. We follow here the method of Kottwitz for convenience only and to make our exposition more efficient, because we then can  quote \cite{RZ} . However, it should be pointed out that this point of view is not very natural in the framework of the present paper, and could be avoided.  This can be done without any additional work  if there is only one prime ${\bf p}$ over $p$.

A more sophisticated alternative proof of  Lemma \ref{uniquenessin prod} uses  the finite subset $B(G, \mu)$ of $B(G)$, for the conjugacy class of cocharacters $\mu$ associated to the conjugacy class of \eqref{defofh}. Then $[b]\in B(G, \mu)$, by Mazur's inequality, cf. \cite{RR}. 
Using the bijections $B(G, \mu)\simeq B(G_{\rm ad}, \mu_{\rm ad})\simeq \prod\nolimits_{{\bf p}}B((G'_{\bold p})_{\rm ad}, \mu'_{{\bf p}, {\rm ad}})$, cf.~\cite{K2}, 6.5., we  obtain a bijection
\begin{equation*}
B(G, \mu)\simeq\prod\nolimits_{{\bf p}}B(G'_{{\bf p}}, \mu'_{{{\bf p}}}) ,
\end{equation*}
where $ \mu'_{{{\bf p}}}$ denotes the minuscule coweight of $G'_{\bold p}$ obtained from $\mu$ via  \eqref{embed}. However, 
Proposition \ref{uniformity} implies that,  for every $\bold p$, the image of $[b]$ in $B(G'_{{\bf p}}, \mu'_{{{\bf p}}})$ is the unique basic element in this set, which implies the assertion of Lemma \ref{uniquenessin prod}.

Let $\bold p$ be of the first type. Then $(G'_{\bold p})_{\rm ad}$ is isomorphic to $(D^\times)_{\rm ad}$, where $D^\times$ denotes the algebraic group over $\BQ_p$ associated to the quaternion division algebra over $\BQ_p$. Furthermore,  the coweight $\mu'_{{\bf p}, {\rm ad}}$ of $(G'_{\bold p})_{\rm ad}$ is given by the unique nontrivial minuscule coweight. It then follows that $B(G'_{{\bf p}}, \mu'_{{{\bf p}}})$ consists only of the unique basic element in this set, cf.~\cite{K2}, \S 6. This gives a proof of Lemma \ref{uniformity} in the style of Kottwitz's view on isocrystals with additional structure and avoids the use of Proposition \ref{uniformity}, cf. the remarks at the end of the Introduction.  The  primes of the second and the third type can also be viewed from this perspective, since for them $\mu'_{{\bf p}, {\rm ad}}$ is central. 
\end{remark}
At this point we want to apply \cite{RZ}, Theorem 6.30. In the notation of that theorem, we want to show that 
$Z = Z'$. Let $I$ be the linear algebraic group over $\BQ$ of loc. cit., i.e.
\begin{equation*}
 I (\BQ) = \{ \alpha \in \End_K (A_0)^{0,\times} \mid \alpha^*(\lambda_0) = c \lambda_0,\, c \in \BQ^\times \}
\end{equation*}
By loc. cit., $I$ is an inner form of $G$.
\begin{lemma}
 The Hasse principle for $I$ is satisfied.
\end{lemma}
\begin{proof}
 This follows from \cite{K}, $\S 7$. Indeed, we are here in the case $A$, for $n=2$, and it is proved in
loc. cit that the Hasse principle is satisfied in the case $A$,  for any even  $n$. 
\end{proof}
We may now apply \cite{RZ}, Theorem 6.30, and obtain an isomorphism
\begin{equation}\label{genunif}
 \Theta : I (\BQ)\bs \CM \times G(\BA^{\infty,p}) / C^p \simeq \CM_{r,h, V} (C^p)^\wedge\, \, \, ,
\end{equation}
with notation as follows. On the RHS appears the formal completion of $\CM_{r,h, V}(C^p)$ along its special fiber 
$\CM_{r,h, V}(C^p) \otimes_{\CCO_{{E_{(p)}}} }\kappa_\nu$. On the LHS, $\CM$ denotes the formal moduli space $\breve{\CM}$
over $\CCO_{\breve{E_\nu}}$ with its Weil descent datum to $\CCO_{E_\nu}$ associated in \cite{RZ} to the data 
$(F \otimes \BQ_p, K \otimes \BQ_p, V \otimes \BQ_p, b, r, \CL)$, attached to the situation at hand.

Let us describe more concretely the formal scheme $\breve{\CM}$, with its Weil descent datum. We enumerate the
prime ideals of $F$ over $p$ as follows:
\begin{equation*}
\begin{aligned}
{\bf p}_1, & \ldots , {\bf p}_r && \, \text{are uniformizing primes of the first kind}\\
{\bf p}_{r+1}, & \ldots , {\bf p}_{r+s} && \, \text{are uniformizing primes of the second kind}\\
{\bf p}_{r+s+1}, & \ldots , {\bf p}_{r+s+t} && \, \text{are uniformizing primes of the third kind .}
\end{aligned}
\end{equation*}
The data $b, \mu$ decompose naturally under the embedding \eqref{embed} as a product $b = \prod_i b_i$ and $ \mu = \prod_i \mu_i$; here $\mu_i$
corresponds to the local generalized CM-type $r_{{\bf p}_i}$.

Let $E_i$ denote the local Shimura field associated to $r_{{\bf p}_i}$. Let $\breve{\CM}_i$ be the formal scheme over
${\rm Spf}\, \CCO_{\breve{E}_i}$ whose values in a scheme $S \in {\rm Nilp}_{\CCO_{\breve{E}_i}}$ are given by the set of isomorphism
classes of quadruples $(X, \iota, \lambda, \varrho),$ where $X$ is a $p$-divisible group over $S$ with an
action $\iota$ of $\CCO_{K_{{\bf p}_i}}$ of CM-type $r_{{\bf p}_i}$ and a polarization $\lambda$ with associated Rosati
involution inducing the non-trivial automorphism of $K_{{\bf p}_i}$ over $F_{{\bf p}_i}$, and where
\begin{equation*}
 \varrho : X \times_S \bar S \lra \BX \times_{{\rm Spec} \, \bar{\kappa}_{{\bf p}_i}} \bar S
\end{equation*}
is a $\CCO_{K_{{\bf p}_i}}$-linear  quasi-isogeny such that $\lambda$ and $\varrho^* (\lambda_\BX)$ differ locally on
$\bar S$ by a scalar in $\BQ_p^\times$.  Here $\bar S$ denotes the special fiber of $S$. It is also assumed that the  kernel of $\lambda$ is trivial if ${\bf p}_i$ is split, 
and is of height $f_{{\bf q}_i}h_{{\bf q}_i}$ and contained in
$X[\iota({\bf q}_i)]$ if ${\bf p}_i$ has a unique prime ideal ${\bf q}_i$ over it. 
The height of $\varrho$ is a locally constant function on
$S$ of the form
\begin{equation*}
 s \mapsto f_{{\bf q}_i} \cdot \breve{c}_i(s) \, ,
\end{equation*}
where $\breve{c}_i : S \lra \BZ$, cf.  \cite{RZ}, Lemma 3.53.

Then it is easily seen that $\breve{\CM}$ is the formal subscheme of
\begin{equation*}
 (\breve{\CM}_1 \times_{{\rm Spf}\, \CCO_{\breve{E}_1}} {\rm Spf}\, \CCO_{\breve{E}_\nu})\times_{{\rm Spf}\, \CCO_{\breve{E}_\nu}} \ldots
\times_{{\rm Spf}\, \CCO_{\breve{E}_\nu}} (\breve{\CM}_{r+s+t} \times_{{\rm Spf}\, \CCO_{\breve{E}_{r+s+t}}}{\rm Spf}\, \CCO_{\breve{E}_\nu})
\end{equation*}
where the functions $\breve{c}_i$ agree.

We now describe the formal schemes $\CM_i$ in more detail. For $i$ with $r+1 \leq i \leq r+s + t$, this has been done in 
\cite{RZ}, 6.46 and 6.48. We record this result as follows.

For $i=r+1, \, \ldots , r+s$, the formal scheme $\breve{\CM}_i$ is the constant \'{e}tale scheme 
$G'_{{\bf p}_i}(\BQ_p)/C_{{\bf p}_i}$ for a certain maximal compact subgroup $C_{{\bf p}_i}$. In these cases $E_i = \BQ_p$,
and the Weil descent datum on $\breve{\CM}_i \simeq G'_{{\bf p}_i}(\BQ_p)/C_{{\bf p}_i}$ is given by multiplication
by a central element $t_i \in G'_{{\bf p}_i}(\BQ_p)$ (equal to $(1,p)$ in the notation cf. \cite{RZ}, Lemma 6.47). We note that in loc.~cit. only the case $h=0$ is considered; however, the result is also valid in the case $h=1$, with the same proof. 

For $i=r+s+1, \, \ldots , r+s+t$, the formal scheme $\breve{\CM}_i$ is again the constant \'{e}tale scheme
$G'_{{\bf p}_i}(\BQ_p)/C_{{\bf p}_i}$ for a certain maximal compact subgroup $C_{{\bf p}_i}$. In these cases, $E_i$ is the 
unramified extension of degree $u$ of $\BQ_p$, where $u$ is an \emph{even} divisor of $2 \cdot f_{{\bf p}_i}$. The Weil 
descent datum on $\breve{\CM}_i \simeq G'_{{\bf p}_i}(\BQ_p)/C_{{\bf p}_i}$ is given by multiplication by a central
element $t_i \in G'_{{\bf p}_i}(\BQ_p)$ (equal to $p^{u/2}$ in the notation of \cite{RZ}, 6.48.)

Now let us consider the uniformizing primes of the first kind, so $1\leq i \leq r$. Then $G'_{{\bf p}_i}$ is the group of unitary similitudes for the hermitian vector space $V_{{\bf p}_i}$
of dimension $2$ over the quadratic extension $K_{{\bf p}_i}$ of $\BQ_p$, and $E_i = \BQ_p$. The height of the 
quasi-isogeny $\varrho$ in the quadruple $(X, \iota, \lambda, \varrho)$ defines the function
\begin{equation*}
 \breve{c}_i : \breve{\CM}_i \lra \BZ \, .
\end{equation*}
It is easily seen that ${\rm Im}(\breve{c}_i) = 2 \BZ$.

Let $\CM_0$ be the Drinfeld moduli scheme. We recall its definition. Let $B$ be the quaternion division algebra
over $\BQ_p$, and denote by $\CCO_B$ its maximal order. Then $\CM_0$ is the formal scheme $\breve{\CM_0}$ over 
${\rm Spf} \, \breve{\BZ}_p$ with its Weil descent datum to ${\rm Spf} \, \BZ_p$ (cf.  \cite{RZ}, Theorem 3.72), where
the value of $\breve{\CM}_0$ on a scheme $S \in {\rm Nilp} \, _{\breve{\BZ}_p}$ is the set of isomorphism classes of
triples $(X, \iota_B, \varrho)$, where $(X, \iota_B)$ is a special formal $\CCO_B$-module on $S$ and where
$\varrho : X \times_S \bar{S} \rightarrow \BX \times_{{\rm Spf} \, \bar{\BF}_p}\bar{S}$ is a $\CCO_B$-linear quasi-isogeny.
Here $(\BX, \iota_B)$ is a fixed \emph{framing object} over $\bar{\BF}_p$. Again the height of $\varrho$ defines
a function
\begin{equation*}
 \breve{c}_0 : \breve{\CM}_0 \lra \BZ \, ,
\end{equation*}
with ${\rm Im} (\breve{c}_0) = 2\BZ$. For $k \in \BZ$, let
\begin{equation}
 \breve{\CM}_i [k] = \breve{c}^{-1}_i (2k) \, , \quad i = 0,1, \, \ldots, r \, .
\end{equation}
Then $ \breve{\CM}_0 [0] \simeq \breve{\Omega}^2_{\BQ_p}$, where $\breve{\Omega}^2_{\BQ_p} = \wh{\Omega}^2_{\BQ_p} \times_{{\rm Spf} \, \BZ_p} {\rm Spf} \, \breve{\BZ}_p$ 
({\it Drinfeld's isomorphism}).
Furthermore
\begin{equation*}
 \breve{\CM}_0 \simeq \breve{\CM}_0 [0] \times \BZ \, ,
\end{equation*}
via the isomorphisms $\breve{\CM}_0 [0] \rightarrow \breve{\CM}_0 [k]$ for any $k$, given by 
$$(X, \iota_B,  \varrho) \longmapsto (X, \iota_B\circ{\rm int}(\Pi)^{-k},\iota_B(\Pi)^{k} \circ \varrho)$$ in the notation of loc. cit.
Under the isomorphism
$ \breve{\CM}_0 \simeq \breve{\Omega}^2_{\BQ_p} \times \BZ$, 
the Weil descent datum on the LHS is given on the RHS by the composite of the natural descent datum on 
$\breve{\Omega}^2_{\BQ_p}$ and translation by $1$, cf. \cite{RZ}, Theorem 3.72.

For $1\leq i\leq r$, we embed $K_{{\bf p}_i}$ into $B$. More precisely,

\smallskip

a) when $K_{{\bf p}_i} / \BQ_p$ is unramified, we write $\CCO_{K_{{\bf p}_i}} = \BZ_p[\delta]$, where 
$\delta^2 \in \BZ^\times_p$ and we choose the uniformizer $\Pi$ of $\CCO_B$ such that $\Pi$ normalizes 
$K_{{\bf p}_i}$ and satisfies $\Pi^2 = p$.

\smallskip

b) when $K_{{\bf p}_i} / \BQ_p$ is ramified, we assume $p\neq 2$. Then we choose for $\Pi$ a uniformizer of
$\CCO_{K_{{\bf p}_i}}$ such that $\Pi^2 \in \BZ_p$.

\smallskip

In \cite{KR} we define for any $i$ with $1\leq i\leq r$ an isomorphism
\begin{equation}\label{newint}
 \breve{\CM}_0 [0] \lra \breve{\CM}_i[0] \, .
\end{equation}
More precisely, we associate to the framing object $(\BX, \iota_\BX)$ of $\CM_0$ a framing object 
$(\BX,\iota, \lambda_\BX)$ of $\CM_i$, where $\iota$ is  the restriction from $\CCO_B$ to $\CCO_{K_{{\bf p}_i}}$ of $\iota_\BX$,
and where $\lambda_\BX$ is a carefully chosen quasi-polarization of 
$\BX$. Similarly,  we associate  to any point $(X, \iota_B, \varrho)$ of $ \breve{\CM}_0[0]$ the point $(X, \iota, \lambda_X, \varrho)$ of $\breve{\CM}_i[0]$,
where $\iota = \iota_B \vert\CCO_{K_{{\bf p}_i}}$ and where $\lambda_X$ is a carefully chosen quasi-polarization
of $X$.
\begin{lemma}\label{comp.drinfeld}
 The isomorphism \eqref{newint} extends to an isomorphism compatible with the Weil descent data,
\begin{equation}\label{extofKR}
 \breve{\CM}_0 \lra \breve{\CM}_i \, .
\end{equation}
\end{lemma}
\begin{proof}
 Let $k \in 2\BZ$ and let $(X, \iota_B\circ{\rm int}(\Pi^{-k}), \iota_\BX(\Pi)^{k} \circ \varrho)$ be a point of $\breve{\CM}_0[k]$, where
$(X, \iota_B, \varrho) \in \breve{\CM}_0[0]$. If $K_{{\bf p}_i} / \BQ_p$ is ramified, we map this point 
to $(X, \iota_X, \lambda_X, \iota_\BX(\Pi)^{k} \circ 
\varrho) \in \breve{\CM}_i[k]$. This makes sense since in this case 
\begin{equation*}
 \iota_\BX(\Pi)^{*} (\lambda_\BX) = -p \cdot \lambda_\BX \, .
\end{equation*}
If $K_{{\bf p}_i} / \BQ_p$ is unramified, we map this point to $(X, \iota_X, \lambda_X, \iota_\BX(\Pi\delta)^{k}
\circ \varrho) \in \breve{\CM}_i[k]$. This makes sense since in this case 
\begin{equation*}
 \iota_\BX(\Pi\delta)^{*} (\lambda_\BX) = {\rm Nm} \, (\delta) \cdot p \cdot \lambda_\BX \, .
\end{equation*}
It follows from the definitions that this defines an isomorphism compatible with the descent
data on $\breve{\CM}_0$ and $\breve{\CM}_i$. 
\end{proof}
\begin{corollary} 
\label {isowithdegree}
Let $i$ with $1\leq i\leq r$. Assume $p\neq 2$ if ${\bf p}_i$ is ramified in $K$. 
 There is an isomorphism
\begin{equation*}\label{compiso}
 \breve{\CM}_i \simeq \breve{\Omega}^2_{\BQ_p} \times \BZ \, ,
\end{equation*}
such that the Weil descent datum on the LHS corresponds to the composite of the natural descent datum on 
$\breve{\Omega}^2_{\BQ_p}$ and translation by $1$.\qed
\end{corollary}
We leave it to the reader to check that (again for $i$ with $1\leq i\leq r$) the group $J_i(\BQ_p)$ of $\CCO_{K_{{\bf p}_i}}$-linear self-isogenies
of $\BX$ which preserve $\lambda_\BX$ up to a scalar in $\BQ^\times_p$ can be identified with the group
of unitary similitudes of the \emph{split} hermitian space of dimension $2$ over $K_{{\bf p}_i}$, and that the
action of an element $g \in J_i(\BQ_p)$ on the LHS of Corollary \ref{isowithdegree}, given by $(X, \iota, \lambda,\rho)\mapsto (X, \iota, \lambda,g\circ\rho)$, is given on the RHS   by 
\begin{equation*}
 (g_{\rm ad}  \, , \text{translation by } {\scriptstyle\frac12}\,\text{\rm ord}\, c(g))\, ,
\end{equation*}
where  $g_{\rm ad} $ is considered as an element in ${\rm PGL}_2(\BQ_p)$, via a chosen isomorphism $(J_{i})_{\rm ad}(\BQ_p) = {\rm PGL}_2(\BQ_p)$. 

We refer to \cite{RZ}, 6.46, resp. 6.48 for a description of the analogous groups $J_i(\BQ_p)$ for $i$ 
with $r+1\leq i\leq r+s$, resp. $r+s+1\leq i\leq r+s+t$. Let $J(\BQ_p)$ be the group of automorphisms of the isocrystal $N$ which are $K$-linear and preserve the anti-symmetric form up to a scalar in $\BQ_p$. 

As in \cite{RZ}, Proposition 6.49, one checks the following proposition. 
\begin{proposition}\label{puttog}
 $J(\BQ_p)$ is the inverse image of the diagonal under the map
\begin{equation*}
 \prod{c_i} : \prod^{r+s+t}_{i = 1} J_i (\BQ_p)\lra \prod^{r+s+t}_{i = 1} \BQ_p^\times\, .
\end{equation*}
Similarly, $G(\BQ_p)$ is the inverse image of the diagonal under
\begin{equation*}
 \prod{c_i} : \prod^{r+s+t}_{i = 1} G'_{{\bf p}_i}(\BQ_p) \lra \prod^{r+s+t}_{i = 1} \BQ_p^\times \, .
\end{equation*}
For $1\leq i\leq r$, let $C_{{\bf p}_i}$ be the unique maximal compact subgroup of $G'_{{\bf p}_i}(\BQ_p)$, and, for $r+1\leq i\leq r+s+t$, let  $C_{{\bf p}_i}$ be 
the maximal compact subgroup of $G'_{{\bf p}_i}(\BQ_p)$ introduced above. The actions of $J_i(\BQ_p)$ on $G'_{{\bf p}_i}(\BQ_p)/C_{{\bf p}_i}$ combine to
give  an action of $J(\BQ_p)$ on 
$G(\BQ_p)/C_p$, where $C_p=G(\BQ_p)\cap \prod C_{{\bf p}_i}$. In case $p=2$, assume that ${\bf p}_i$ is unramified in $K$, for $1\leq i\leq r$. 

There is a $J(\BQ_p)$-equivariant isomorphism of formal schemes
\begin{equation*}
 \breve{\CM} \simeq \big(\prod^r_{i = 1} \breve{\Omega}^2_{\BQ_p}\big) \times G(\BQ_p) / C_p .
\end{equation*}
The action of $J(\BQ_p)$ on the first $r$ factors is via the projections $J(\BQ_p) \lra J_i(\BQ_p) \lra
\PGL_2(\BQ_p)$,
and on the last factor is as described above.

The localization $E_\nu$ of the reflex field $E$ is the composite of the fields $E_i$, for $i=r+1,\ldots,r+s+t$ described above. The Weil descent datum on $\breve{\CM}$ relative to $\breve{E}_\nu/E_\nu$ induces on the RHS the natural descent
datum on the first $r$ factors multiplied with the action of an element $g \in G(\BQ_p)$ on the last factor,
where $g$ maps to the element $(t_1, \, \ldots , t_{r+s+t}) \in \prod G'_{{\bf p}_i}(\BQ_p)$, where 
$t_{r+1}, \, \ldots , t_{r+s+t}$ are the central elements described above, and where for $i = 1, \, \ldots, r$
the element $t_i$ is any element with ${\rm ord}\,c_i(t_i) = 1$.\qed
\end{proposition}
Using this proposition, we obtain as in \cite{RZ} the following theorem.  When $p=2$, we make the usual assumption on prime ideals $\bold p\vert p$ of the first kind. 
\begin{theorem}\label{intunif}
 Let $C=C^p\cdot C_p$, where $C_p$ is defined in Proposition \ref{puttog}. 
 There is a $G(\BA^{\infty, p})$-equivariant isomorphism of formal schemes over $\CCO_{\breve{E}_\nu} = \breve{\BZ}_p$,
$$
I(\BQ)\bs [ (\prod^r_{i=1} \breve{\Omega}^2_{\BQ_p}) \times G(\BA^{\infty})/C ] \simeq \CM_{r,h,V}(C^p)^{\wedge}
\times_{{\rm Spf}\CCO_{E_\nu}} {\rm Spf}\CCO_{\breve{E}_\nu} .
$$ The group $I$ is the inner form of $G$, unique up to isomorphism,  such that
$I_{\rm {ad}}(\BR)$ is compact, and $I(\BQ_p)$ is the group $J(\BQ_p)$ defined above and such that 
$I(\BA^{\infty, p}) \simeq G(\BA^{\infty, p})$. The natural descent datum on the RHS induces on the LHS the composite of the natural
descent datum on the first $r$ factors multiplied with the action of the element 
$g \in G(\BQ_p) \subset G(\BA^{\infty})$ in Proposition \ref{puttog} on the last factor.\qed
\end{theorem}

\begin{corollary}
Under the conditions of the previous theorem, $\CM_{r,h,V}(C^p)\otimes_{{\CCO_E}_{(p)}}\CCO_{{E}_\nu}$ is flat over $\Spec\,\CCO_{{E}_\nu}$. 
\end{corollary}
\begin{proof}
Indeed, flatness  holds for the LHS in the last theorem.  
\end{proof}
As a special case of the previous theorem, we formulate the following corollary.
\begin{corollary}\label{onlytype1}
 Let $K/F$ be a CM-extension of the totally real field $F$ of degree $d$ over $\BQ$. Let $p$ be a prime number
that decomposes completely in $F$ and such that each prime divisor ${\bf p}$ of $p$ is inert or ramified in $K$. We also
assume that if $p=2$, then no such ${\bf p}$ is ramified. Let $V$ be a hermitian vector space of dimension $2$ over $K$
with signature $(1, 1)$ at every archimedean place of $K$. We also assume that $\inv_{{\bf p}}(V)=-1\, , 
\text{ for all } {\bf p}\vert p$. 

Let $G$ be the group of unitary similitudes of $V$ with multiplier in $\BQ$. Let $C^p$ be an open compact subgroup of $G(\BA^{\infty, p})$, and let $C=C^p\cdot C_p$, where $C_p$ is the unique maximal compact subgroup of $G(\BQ_p)$. Let ${\rm Sh}_C$ be the corresponding Shimura variety, which is defined over $\BQ$.

Let $\CM_{r, h, V}(C^p)$ be the model of ${\rm Sh}_C$ over $\BZ_{(p)}$ which parametrizes almost principal CM-triples
$(A, \iota, \lambda)$ of type $(r, h)$ with level-$C^p$-structure, where $r_{\varphi} = 1, \text{ for all } \varphi$, and with ${\rm inv}_v (A, \iota, \lambda)={\rm inv}_v(V), \text{ for all } v$.  We assume for any prime divisor $\bold p$ of $p$ that $\htt_{\bold p}=0$, resp. $\htt_{\bold p}=1$, when  $\bold p$ is ramified, resp. inert.  We also assume the following compatibility condition between $h$ and the invariants of $V$, cf. Proposition \ref{prop.nonempty}:
\begin{itemize}
\item  If $\htt_{\bold p_v}=0$ and $v$ is inert in $K/F$, then $\inv_v(V)=1$. 
\item  If $\htt_{\bold p_v}=2$, then $\inv_v(V)=1$. 
\item If $\htt_{\bold p_v}=1$, then $v$ is inert in $K/F$ and $\inv_v(V)=-1$. 
\end{itemize}
 Then there is a 
$G(\BA^{\infty, p})$-equivariant isomorphism of $p$-adic formal schemes
\begin{equation*}
 I(\BQ)\bs\big[\big((\widehat{\Omega}^2_{\BQ_p})^d \times_{{\rm Spf}\,\BZ_p}{\rm Spf}\,\breve{\BZ}_p\big) \times
 G(\BA^{\infty})/C\big] \simeq \CM_{r,h,V}(C^p)^{\wedge} \times _{{\rm Spf}\,\BZ_p}{\rm Spf}\,\breve{\BZ}_p \, .
\end{equation*}
Here $I(\BQ)$ is the group of $\BQ$-rational points of the inner form $I$ of $G$ such that $I_{\rm ad} (\BR)$ is compact, and
$I_{\rm ad}(\BQ_p) \simeq {\rm PGL}_2(\BQ_p)^d, \text{and } I(\BA^{\infty, p}) \simeq G(\BA^{\infty, p})$.

The natural descent datum on the RHS induces on the LHS the natural descent datum on the first factor multiplied
with the translation action of
$ (1, t)$ on $ G(\BA^{\infty, p})/C^p \times G(\BQ_p)/C_p$, where $t\in G(\BQ_p)$ is any element with ${\rm ord}\,c(t) = 1$.\qed
\end{corollary}

We end this section by  showing  an integral version of Theorem \ref{ThmA} in the introduction. It shows that in the special case $F=\BQ$, the theory of local invariants of CM-triples of section \ref{section3} is not needed. We use the notation introduced before the statement of Theorem \ref{ThmA}. We consider the  DM-stack 
$\CM_h(C^p)$ over $\Spec\,\BZ_{(p)}$ which parametrizes quadruples $(A,\iota,\lambda,\eta^p)$, where $(A, \iota, \lambda)$ is an almost principal CM-triple of type $(r,h)$ where $r_\varphi=1 \text{ for all }\varphi$,   and where $h$ satisfies the usual compatibility condition of Proposition \ref{prop.nonempty} and such that  $h_{p}=1$, resp. $h_{ p}=0$, according as  $p$ is unramified or ramified in the quadratic extension $K$ of $\BQ$. Finally, $\eta^p$ is a level-$C^p$-structure.  The next theorem implies Theorem \ref{ThmA}. We make the usual assumption when $p=2$. 
\begin{theorem}There are  canonical isomorphisms between schemes over $\Spec \BQ$,
\begin{equation*}
\CM_h(C^p)\otimes_{\BZ_{(p)}}\BQ  \simeq \, {\rm Sh}_C ,
\end{equation*}
and  between schemes over $\Spec\, \breve\BZ_p$,
\begin{equation*} 
\CM_h(C^p) \otimes_{\BZ_{(p)}}\breve\BZ_p \simeq \, \big(\bar G(\BQ)\bs [\wh{\Omega}^2_{\BQ_p} \times  G(\BA^{\infty})/ C]\big)\otimes_{\BZ_p}\breve\BZ_p.
\end{equation*}
\end{theorem}
\begin{proof}

All we have to prove is that $\CM_h(C^p)=\CM_{r, h, V}(C^p)$, i.e., that $\inv_v(A,\iota, \lambda)=\inv_v(V),\, \forall v$. 
 This is obvious if $(A,\iota, \lambda)$ is over a field of characteristic zero, because $\inv_v(A,\iota, \lambda)=\inv_v(V)$ for all places $v\neq p$ of $\BQ$ by the existence of the level structure $\eta^p$, and by the product formula for $\inv_v(A,\iota,\lambda)$ and for $\inv_v(V)$, cf. Proposition \ref{prod.formula}, (i). Now let $(A,\iota, \lambda)$ be over a field of characteristic $p$. But the moduli space $\CM_h(C^p)$ is flat over $\Spec\,\BZ_{(p)}$, as follows from the theory of local models. Indeed, the ramified case is  covered by the theorem of Pappas, \cite{P}, Thm. 4.5., b), cf. also \cite{PRS}, Remark 2.35. The unramified case is covered by the theorem of G\"ortz, cf. \cite{G}, Thm. 4.25, cf.  also \cite{PRS}, Thm. 2.16. Hence we may apply Proposition \ref{prod.formula}, (ii) to deduce that the product formula for $\inv_v(A,\iota, \lambda)$ is also valid in this case, and hence the same argument implies the claim. 
\end{proof}

\begin{remark}\label{flatness-deg}
 It is the flatness issue that prevents us from allowing primes ${\bf p}$ such that $K_{\bf p}/F_{\bf p}$ is ramified in the definition of uniformizing primes of the second kind. Indeed, the naive local models  
are known to be non-flat in this situation, cf. \cite{PRS}. By imposing more conditions on $(A, \iota, \lambda)$, it should be possible to formulate a moduli problem that is flat; this would yield more general situations in which integral uniformization holds. 

\end{remark}

\section{Rigid-analytic uniformization}

In this section, we will consider rigid-analytic uniformization. This is weaker than integral uniformization, but more general in two ways. First, since flatness of integral models is no issue, we are able to also allow {\it degenerate} CM-types at ramified primes above $p$, comp. Remark \ref{flatness-deg}. 
Secondly, we can allow for level of the form $C=C^pC_p$, where  $C_p$ is allowed to be strictly smaller than a maximal compact subgroup of $G(\BQ_p)$. The latter variant is inspired by the corresponding rigid-analytic uniformization theorem of Drinfeld, cf.  \cite{BC, Dr}. 

We first formulate a rigid-analytic version of Theorem \ref{intunif}. We now allow also primes of the fourth kind, extending the list from Definition \ref{defunifprimes}.

\begin{definition}
Let $K/F$ be a {\it ramified} quadratic extension of $p$-adic fields. We call $(K/F, r, h, \ep)$ {\it uniformizing data of the fourth kind}, if $r_\varphi\in \{0, 2\}$ for all $\varphi\in \Hom(K, \bar\BQ_p)$, and $h=0$ and $\ep=\pm 1$. 
\end{definition}

In the notation of section \ref{Integral uniformization}, we continue the enumeration of the prime ideals of $F$ over $p$  by also allowing prime ideals ${\bf p}_{r+t+s+1},\ldots, {\bf p}_{r+s+t+u}$ of the fourth kind. Then the rigid-analytic space $\breve\CM^{\rm rig}$ associated to $\breve\CM$ is the subspace of 
\begin{equation*}
 (\breve{\CM}_1^{\rm rig} \times_{{\rm Sp}\, {\breve{E}_1}} {\rm Sp}\, {\breve{E}_\nu})\times_{{\rm Sp}\, {\breve{E}_\nu}} \ldots
\times_{{\rm Sp}\, {\breve{E}_\nu}} (\breve{\CM}_{r+s+t+u}^{\rm rig} \times_{{\rm Sp}\, {\breve{E}_{r+s+t+u}}}{\rm Sp}\, {\breve{E}_\nu})
\end{equation*}
where the functions $\breve{c}_i$ agree. Now since the CM-type for primes of the fourth kind is degenerate, all points of the rigid spaces $\breve{\CM}_{i}^{\rm rig} $ map to the same point under the {\it period map}, for  $i=r+s+t+1,\ldots, r+s+t+u$, cf. \cite{RZ}. Therefore, these points all lie in one single isogeny class, and are classified by their $p$-adic Tate module. The set of these Tate modules is a homogeneous space under $G'_{{\bf p}_{i}}(\BQ_p)$. In fact, denoting by $C_{{\bf p}_{i}}$ the stabiliser of a self-dual lattice in $V_{{\bf p}_{i}}$, we can identify $\breve{\CM}_{i}^{\rm rig} $ with the discrete space $G'_{{\bf p}_{i}}(\BQ_p)/C_{{\bf p}_{i}}$. Furthermore,  $J_{i}(\BQ_p)=G'_{{\bf p}_{i}}(\BQ_p)$. Let $\Phi_i$ be the subset of $\Hom_{\BQ_p}(K_{{\bf p}_i}, \bar\BQ_p)$ where the value of $r_i$ is equal to zero. The  reflex field $E_i$ is characterized by 
\begin{equation*}
\Gal(\bar\BQ_p/E_i)=\{ \sigma\in\Gal(\bar\BQ_p/\BQ_p)\mid \sigma(\Phi_i)=\Phi_i \}\, .
\end{equation*}
The Weil descent datum is given by translation by an element in the center of $G'_{{\bf p}_{i}}(\BQ_p)$ on $G'_{{\bf p}_{i}}(\BQ_p)/C_{{\bf p}_{i}}$. This central element can be deduced from the reciprocity map attached to $G'_{{\bf p}_{i}}$ and $r_i$. 

We may now formulate the following analogue of Theorem \ref{intunif}. The assumptions are the same, except that we now also accept uniformizing primes of the fourth kind. The proof is completely analogous.

\begin{theorem}
 Let $C=C^p\cdot C_p$,  where $C_p=G(\BQ_p)\cap \prod\nolimits_{i=1}^{r+s+t+u} C_{{\bf p}_i}$, and define $J(\BQ_p)$ and its action on $G(\BQ_p)/C_p$ in analogy with Proposition \ref{puttog}. In case $p=2$, assume that $p$ is unramified in $K_{{\bf p}_i}$, for $1\leq i\leq r$.  
 There is a $G(\BA^{\infty, p})$-equivariant isomorphism of rigid-analytic spaces  over ${\breve{E}_\nu} $,
$$
I(\BQ)\bs \big[ \big(\prod^r_{i = 1}{\Omega}^2_{\BQ_p}\times_{{\rm Sp\,} {\BQ_p}} {\rm Sp\,} {\breve{E}_\nu}\big) \times G(\BA^{\infty})/C \big ] \simeq \CM_{r,h,V}(C^p)^{\rm rig}
\times_{{\rm Sp\,} {E_\nu}} {\rm Sp\,} {\breve{E}_\nu} .
$$ The group $I$ is the inner form of $G$, unique up to isomorphism,  such that
$I_{\rm {ad}}(\BR)$ is compact, and $I(\BQ_p)$ is the group $J(\BQ_p)$ defined above and such that 
$I(\BA^{\infty, p}) \simeq G(\BA^{\infty, p})$. The natural descent datum on the RHS induces on the LHS the composite of the natural
descent datum on the first $r$ factors multiplied with the action of an element 
$g \in G(\BQ_p) \subset G(\BA^{\infty})$, described in Proposition \ref{puttog} and above, on the last factor.\qed
\end{theorem}

As a special case, we formulate a corollary in the style of Theorem \ref{ThmA}.
\begin{corollary}\label{rigunif} Let $p$ decompose completely in the totally real field $F$ of degree $d$ over $\BQ$, and let $K/F$ be a CM quadratic extension such that each prime divisor $\bf p$ of $p$ in $F$ is inert or ramified in $K$. We also assume that if $p=2$, then no $\bf p$ is ramified.  Let $V$ be a hermitian vector space of dimension $2$ over $K$
with signature (1, 1) at every archimedean place of $F$. We also assume that $\inv_{{\bf p}}(V)=-1$ for all 
${\bf p} \vert p$. 
Let $G$ be the group of unitary similitudes of  $V$ with multiplier in $\BQ^\times$. 
Let $C^p$ be an open compact subgroup of $G(\BA^{\infty, p})$, and let $C=C^p\cdot C_p$, where $C_p$ is the unique maximal compact subgroup of $G(\BQ_p)$. 
Let ${\rm Sh}_C$ be the canonical model of the corresponding Shimura variety, a projective variety of dimension $d$ defined over $\BQ$. 

 There is a 
$G(\BA^{\infty, p})$-equivariant isomorphism of projective schemes over $\breve \BQ_p$, 
\begin{equation*}
{\rm Sh}_C \otimes_\BQ \breve{{\BQ}}_p
\simeq \big(I(\BQ)\bs [({\Omega}^2_{\BQ_p})^d \times  G(\BA^{\infty})/ C]\big)\otimes_{\BQ_p}\breve{{\BQ}}_p
\end{equation*}
Here $I(\BQ)$ is the group of $\BQ$-rational points of the inner form $I$ of $G$ such that $I_{\rm ad} (\BR)$ is compact, 
$I_{\rm ad}(\BQ_p) \simeq {\rm PGL}_2(\BQ_p)^d, \text{and } I(\BA^{\infty, p}) \simeq G(\BA^{\infty, p})$.

\end{corollary}

We next allow deeper level structures at $p$.

We first point out  a further variant of the moduli space introduced in Remark \ref{p-variantofmoduli}. Namely, denoting by $C_M$ the stabilizer of
the fixed lattice $M$ in $V$, we may introduce, for any open compact subgroup $C$ of $G(\BA^\infty)$ which is contained with finite index in $C_M$,  the stack $\CM_{r, h, V}( C)$ over $\Spec E$ 
which, in 
addition to $(A, \iota, \lambda)$ satisfying conditions (\ref{ker.rank}) and (\ref{inv.condition}), fixes a level structure ${\rm mod} \, C$, i.e.,  an isomorphism compatible with $\iota$ and $\lambda$
\begin{equation*}
 \wh T (A)  \simeq M\otimes\widehat{\BZ} \, \, {\rm mod} \, C\, ,
\end{equation*}
in the sense of Kottwitz \cite{K}. Note that, due to the existence of the level structure, the condition \eqref{inv.condition} is automatic. If $C=C^pC_p^0$, where $C_p^0$ denotes the maximal compact subgroup of $G(\BQ_p)$ occurring in the previous theorem, then by Lemma \ref{lattice.lemma}, $\CM_{r, \htt, V} (C)=\CM_{r, \htt, V}( C^p)\otimes_{\CCO_{E_{(p)}}}E$.

Recall the formal moduli space $\breve{\CM}$ over $\Spf\, O_{\breve E_\nu}$ of the beginning of this section;  in particular,  we allow uniformizing primes of the fourth kind. We denote by $(X, \iota, \lambda)$ the universal object over $\breve{\CM}$. Let $\breve{\CM}^{\rm rig}$ be the rigid-analytic space  over ${\rm Sp}\, \breve E$ associated to the formal scheme $\breve{\CM}$. For any open compact subgroup $C_p$ contained in the maximal
compact subgroup $C_p^0$, we may consider the rigid-analytic space $\breve\BM_{C_p}$ which trivializes the local system $T_p(X)$ on 
$\breve{\CM}^{\rm rig}$,
\begin{equation*}
T_p(X)\simeq M\otimes \BZ_p \, {\rm mod}\,  C_p \, .
\end{equation*}
Here $M$ is again a fixed lattice in $V$, as in Remark \ref{p-variantofmoduli}, and the isomorphism is, of course, supposed to be $O_{K\otimes {\BQ_p}}$-linear and to preserve the symplectic forms up to a scalar in $\BZ_p^\times$. The rigid space $\breve\BM_{C_p}$  comes with a Weil descent datum to $\Sp\,  E_\nu$. 

Comparing now the $p$-primary level structures of an abelian variety and of its associated $p$-divisible group, we obtain the following rigid-analytic uniformization theorem.
\begin{theorem}\label{thm.rig}
Let $C$ be of the form $C^p C_p$, where $C_p\subset C_p^0$. Then there exists a $G(\BA^\infty)$-equivariant isomorphism of rigid-analytic spaces over $\Sp\, \breve E_\nu$, compatible with the Weil descent data on both sides, 
\begin{equation*}
I(\BQ)\backslash \breve{\BM}_{C_p}\times G(\BA^{\infty, p})/C^p\simeq \CM_{r, h, V} (C)^{\rm rig}\times_{{\rm Sp}\, E_\nu} {{\rm Sp}\,  \breve{E}_\nu}\ . 
\end{equation*}\qed
\end{theorem}
Again, as for the integral version of the uniformization theorem, one can make this statement more explicit. For this, we suppose that $C_p$ is of the form 
\begin{equation}
C_p=G(\BQ_p)\cap\prod_{i=1}^{r+s+t+u}C_{\bold p_i} \, ,
\end{equation}
where $C_{\bold p_i}$ are open compact subgroups in $G'_{\bold p_i}(\BQ_p)$ contained in $C_{\bold p_i}^0$. (Note that further above, $C_{\bold p_i}^0$ was denoted by $C_{\bold p_i}$.) We consider the corresponding finite etale coverings $\breve\BM_{C_{\bold p_i}}$ of $\breve \CM_i^{\rm rig}$, for $i=1,\ldots,r+s+t+u$, i.e., those trivializing ${\rm mod}\, C_{\bold p_i}$ the $p$-adic Tate modules of the universal objects over $\breve \CM_i^{\rm rig}$. For $i=r+1,\ldots,r+s+t+u$, $\breve\BM_{C_{\bold p_i}}$ is the constant etale scheme $G'_{\bold p_i}(\BQ_p)/C_{\bold p_i}$, with the Weil descent datum described in the previous section. It is equipped with the obvious morphism to $\BQ_p^\times/c_i(C_{\bold p_i}) $. 

Now for $i$ with $1\leq i\leq r$, and with the usual assumption when $p=2$, we have an   isomorphism 
\begin{equation*}
\breve{\CM}_i \simeq \big(\wh{\Omega}^2_{\BQ_p}\times_{{\rm Spf\,}\BZ_p} {{\rm Spf\,}\breve \BZ_p}\big)\times \BZ =\big(\wh{\Omega}^2_{\BQ_p}\times_{{\rm Spf\,}\BZ_p}{\rm Spf\,}\breve \BZ_p\big) \times G'_{\bold p_i}(\BQ_p)/C^{0}_{\bold p_i}
\end{equation*}
as in Corollary \ref{isowithdegree}, which  induces  isomorphisms of rigid-analytic spaces 
\begin{equation}\label{formalprep}
\breve{\CM}_i ^{\rm rig}\simeq \big({\Omega}^2_{\BQ_p} \times_{{\rm Sp}\, \BQ_p}{{\rm Sp}\, \breve\BQ_p}\big)\times\BZ\simeq \big({\Omega}^2_{\BQ_p} \times_{{\rm Sp}\, \BQ_p}{{\rm Sp}\, \breve\BQ_p}\big)\times G'_{\bold p_i}(\BQ_p)/C^{0}_{\bold p_i}\, .
\end{equation}
 The covering space $\breve\BM_{C_{\bold p_i}}$  of  $\breve{\CM}_i ^{\rm rig}$ maps  to the discrete space $\BQ_p^\times/c_i(C_{\bold p_i})$, covering the second projection of $\breve{\CM}_i ^{\rm rig}$ to $\BZ\simeq \BQ_p^\times/c_i(C^0_{\bold p_i})$ in \eqref{formalprep}, comp. also \cite{Chen}.   
 
\begin{comment}

\begin{remark}
We note that the covering spaces   $\breve\BM_{C_{\bold p_i}}$ of $\breve{\CM}_i ^{\rm rig}$ can be identified with certain  members of the Drinfeld  tower under the isomorphism $\breve{\CM}^{\rm rig}_0 \lra \breve{\CM}^{\rm rig}_i$ of Lemma \ref{comp.drinfeld}. {\bf CAN ONE MAKE THIS MORE EXPLICIT?} In this sense, one may consider these covering spaces as `known'. 
\end{remark}

\end{comment}

  We now obtain the following description of the covering space $\breve{\BM}_{C_p}$, which makes Theorem \ref{thm.rig} more explicit.
\begin{theorem} The covering space $\breve{\BM}_{C_p}$ can be identified with  the inverse image of $\BQ_p^\times/c(C_p)\subset \prod^{r+s+t+u}_{i = 1} \BQ_p^\times/c_i(C_{\bold p_i})$  under the morphism of rigid-analytic spaces
\begin{equation*}
\prod^{r}_{i = 1} \breve\BM_{C_{\bold p_i}}\times \prod^{r+s+t+u}_{i=r+1}G'_{\bold p_i}(\BQ_p)/C_{\bold p_i} \lra \prod^{r+s+t+u}_{i = 1} \BQ_p^\times/c_i(C_{\bold p_i}) \, .
\end{equation*}\qed
\end{theorem}

\section{Appendix: Twisted unitary similitude groups}\label{append}

In this section we recall the construction of twisted unitary similitude groups 
from Boutot-Zink\footnote{In \cite{BZ} the case of division algebras of arbitrary rank $d^2$ over $K$ is considered. }, \cite{BZ}. 
Throughout this section, $K$ is a CM field with maximal totally real subfield $F$. 

 Let $S$ be a quaternion algebra over $K$ with a positive involution $*$.  
Let $W=S$, viewed as an $S$ bi-module, and let  $\psi: W \times W\rightarrow \BQ$ be a 
non-degenerate alternating form such that $\psi(sw_1,w_2) = \psi(w_1,s^*w_2)$. 
We let $G^\bullet$ be the reductive group over $\BQ$ such that\footnote{We are following \cite{BZ}, but our $G^\bullet$
is slightly smaller due to our condition on the scale factor $\gamma(g)$.} for any $\BQ$-algebra $R$, 
\begin{equation}
\label{firstbullet}
G^\bullet(R) = \{ \ g\in \GL_{S\otimes_\BQ R}(W\otimes_{\BQ}R) \mid \psi(gw_1,g w_2) = \gamma(g)\,\psi(w_1,w_2), 
\gamma(g) \in R^\times\ \}.
\end{equation}
Define another involution $\star$ of $S$ by $\psi(w_1s,w_2) = \psi(w_1,w_2 s^\star)$. Then 
\begin{equation}
\label{secondbullet}
G^\bullet(R) \simeq \{ \ g\in (S^{\text{opp}}\otimes_\BQ R)^\times \mid g g^\star = \gamma(g) \in R^\times\ \}.
\end{equation}

%Next, we relate this group to the ones considered earlier in the paper. 
For a number field $L$, let 
$T_L = R_{L/\BQ}(\BG_m)$, and for a quaternion algebra $B$ over $F$, let $G^B = R_{F/\BQ}(B^\times)$. Both are algebraic groups over $\BQ$. 
\begin{lemma} Suppose that $S$, $*$ and $\psi$ are given. 
There is a unique quaternion algebra $B$ over $F$ such that there is an exact sequence 
$$1\lra T_F\lra (T_K\times G^B)_0 \lra G^\bullet \lra 1 ,$$
where, for a $\BQ$-algebra $R$,  
$$(T_K\times G^B)_0(R) =\{\ (a,b) \in T_K\times G^B(R)\mid {\rm Nm}_{K/F}(a)\cdot \nu(b) \in R^\times\ \}.$$
Here the homomorphism from $T_F$ to $(T_K\times G^B)_0$ is the anti-diagonal embedding $t\mapsto (t, t^{-1})$, and $\nu$ denotes the reduced norm ${\rm Nm}_{B/F}$.\qed 
\end{lemma}

Conversely, suppose that a quaternion algebra $B$ over $F$ and a positive involution $*$ of $S= K\otimes_F B$ are given. 
Denoting by $b\mapsto b^\iota$ the main involution on $B$, define an involution of the second kind on $S$ by $(a\otimes b)^\dagger = \bar a\otimes b^\iota$, and write
$s^* = \text{\rm Ad}(\alpha)(s^\dagger)$ for $\alpha\in S$ with $\alpha^\dagger = \alpha$. Note that 
$B$ is precisely the subalgebra of $S$ fixed by $\dagger\circ\iota$.  Choose an element $\delta\in K^\times$ 
with $\bar\delta = -\delta$, and define
\begin{equation}\label{defpsi}
\psi(w_1,w_2) = \text{\rm tr}_{S/\BQ}(w_1 \delta^{-1}\alpha^{-1} w_2^*),
\end{equation}
where $\text{\rm tr}_{S/\BQ} = \text{\rm tr}_{K/\BQ}\circ \text{\rm tr}_S$, for $\text{\rm tr}_S$ the reduced trace on $S$. 
The form $\psi$ is then a non-degenerate alternating form on $W=S$ with $\psi(sw_1,w_2) = \psi(w_1,s^*w_2)$.
Moreover, $\psi(w_1s,w_2) = \psi(w_1,w_2 s^\dagger)$, and the associated group $G^\bullet$ is then described in terms of $G^B$ and $T_K$, 
as in the lemma.  Thus, we can start with data $B$ and $*$ on $K\otimes_FB$,  instead of $S$, $*$, and $\psi$. 
Note that there is also a form $h:W\times W\lra K$ defined by 
$$h(w_1,w_2) = \text{\rm tr}_S(w_1 \alpha^{-1} w_2^*).$$
This is a hermitian form on $W$ viewed as a left $K$-vector space,  with 
$$\psi(w_1,w_2) = \tr_{K/\BQ} \delta^{-1}\,h(w_1,w_2).$$
In particular, up to our additional restriction on scale factors, 
$G^\bullet$ is the subgroup of similitudes of this hermitian space that commute with the (left multiplication) action of $S$. 

As a special case, suppose that $S= M_2(K)$,  with positive involution $*$ given by  $s^* = {}^t\bar s$.  Let $B$ be a 
quaternion algebra over $F$ that is split by $K$,   and fix an isomorphism $S= M_2(K) \simeq K\otimes_F B$.  
For the idempotent 
$$e = \begin{pmatrix} 1&{}\\{}&0\end{pmatrix}\in S, \qquad e+e^\iota =1,$$
there is a decomposition
$$V = e\cdot S, \qquad V' = e^\iota \cdot S, \qquad S= V \oplus V'$$
as left $K$-modules and right $S$-modules. This decomposition is orthogonal with respect to $\psi$ defined by (\ref{defpsi}), 
and hence also with respect to the hermitian form $h$.  Let $h_V$ be the restriction of $h$ to $V$. 
Let 
$$\rho = \begin{pmatrix}{}&1\\1&{}\end{pmatrix}.$$
Since $e\cdot \rho = \rho\cdot e^\iota$, left multiplication by $\rho$ 
interchanges $V$ and $V'$ and, since $\rho^*=\rho$ and $\rho^2=1$,  defines an isometry of these spaces. 
The group $G^\bullet$ preserves the decomposition $W = V+ V'$ and, since left multiplication by $\rho$ 
commutes with the action of $G^\bullet$, the restriction to $V$ defines an isomorphism 
$$G^\bullet \overset{\sim}{\longrightarrow} R_{F/\BQ} \text{\rm GU}(V)_0,$$
where $\text{\rm GU}(V)_0$ denotes the subgroup where the scale factor lies in $\BQ^\times$. 

%
%Let $K$ be a CM field with maximal totally real subfield $F$, $|F:\BQ|=d$, etc. , as in the main body of the text. 
%Fix an element $\delta\in K^\times$ with $\bar\delta =-\delta$. 

%Let $U$ be a $K$ vector space of of dimension $n$ with a non-degenerate $\BQ$-bilinear alternating form 
%%$$\gs{\ }{\ }: U\times U \lra \BQ,$$
%such that 
%%$$\gs{a x}{y} = \gs{x}{\bar a y},\qquad \forall a\in K.$$

%---------------------------------------------------

\bigskip
\obeylines
Department of Mathematics
University of Toronto
40 St. George St., BA6290
Toronto, ON M5S 2E4, Canada.
email: skudla@math.toronto.edu

\bigskip
\obeylines
Mathematisches Institut der Universit\"at Bonn  
Endenicher Allee 60 
53115 Bonn, Germany.
email: rapoport@math.uni-bonn.de
%xxxxxxxxxxxxxxxxxxxxxxxxxxxxxxxxxxxxxxxxxxxxxxxxxxxxxxxxxxxxxxxxxxxxxxxxxxxxxxxxxxxxxxxxxxxxxxxxxxxxxxxxxxxxxxxxxxxxxxxxxxxxxxxxxxxxxxxxxxxxxxxxxxxxxxxxx
\end{document}